\documentclass{amsart}

\usepackage{color,graphicx,amssymb,latexsym,amsfonts,txfonts,amsmath,amsthm}
\usepackage{pdfsync,enumitem}
\usepackage{amsmath,amscd}
\usepackage[all,cmtip]{xy}

\usepackage{hyperref}
\hypersetup{
    colorlinks=true,       % false: boxed links; true: colored links
    linkcolor=blue,          % color of internal links
    citecolor=blue,        % color of links to bibliography
    filecolor=blue,      % color of file links
    urlcolor=blue           % color of external links
}

\newtheorem{theo}{Theorem}
\newtheorem{prop}{Proposition}
\newtheorem{coro}{Corollary}
\newtheorem{lemma}{Lemma}

\theoremstyle{remark}
\newtheorem{rema}{\bf Remark}

\numberwithin{table}{subsection}

\author{Rub\'en A. Hidalgo and Jhony S. Leyca}
\title{Purely non-free actions on Riemann surfaces}
\thanks{Partially supported by Project ANID-FONDECYT 1230001}


\address{
Departamento de Matem\'atica y Estad\'{\i}stica, Universidad de La Frontera, Casilla 54-D, Temuco,
Chile}
\email{ruben.hidalgo@ufrontera.cl, j.leica01@ufromail.cl}

\begin{document}

%%%%%%%%%%%%%%%%
\begin{abstract}
A group of conformal automorphisms of a Riemann surface is said to act purely non-free, or to produce a gpnf-action,
if each of its finite-order elements has fixed points. 

It is known that every finite group admits a gpnf-action on a suitable closed Riemann surface of genus at least two.
If $G$ is an infinite countable group, and $X$ is an infinite-type orientable surface, without planar ends and whose space of ends is self-similar, then we observe that there is a Riemann surface structure on $X$ admitting a gpnf-action by $G$.

We obtain upper bounds for those finite groups producing a gpnf-action on a closed Riemann surface of genus $g \geq 2$.

If $G$ is a finite group, then its purely non-free genus $\sigma^{pnf}_{hyp}(G)$ is the minimum genus of at least two of a closed Riemann surface admitting a gpnf-action by $G$. This parameter is already known for cyclic, abelian non-cyclic groups, ${\mathcal A}_{4}$, ${\mathcal A}_{5}$ and ${\mathfrak S}_{4}$. For general $G$, we provide the value of $\sigma^{pnf}_{hyp}(G)$, correcting a previous statement on this value. In particular, we observe that every genus is the purely non-free genus of some group. 

As an application, we describe the minimal purely non-free actions for the family of metacyclic groups
$G^{0}_{s}(n)=\langle x,y: x^{n}=y^{2}=1, yxy=x^{s}\rangle,$
where $n \geq 3$ and $2 \leq s \leq n-1$ is such that $s^{2} \equiv 1 \mod(n)$. 
\end{abstract}
%%%%%%%%%%%%%%%%%%%

\maketitle

%%%%%%%%%%%%%%
%%%%%%%%%%%%%%
\section{Introduction}
The origins of the study of groups of conformal automorphisms of Riemann surfaces are due to the works of Riemann, Klein, Schwarz, Hurwitz, Greenberg, among others. 

Let $S$ be a Riemann surface with negative Euler characteristic (so its universal cover space is the hyperbolic half-plane ${\mathbb H}^{2}$, and it has a non-abelian fundamental group), and let ${\rm Aut}(S)$ be its group of conformal automorphisms.
If $S$ is a closed Riemann surface of genus $g \geq 2$, then Schwarz \cite{Schwarz} proved that ${\rm Aut}(S)$ is finite; later Hurwitz \cite{Hurwitz} proved that $|{\rm Aut}(S)| \leq 84(g-1)$.
If $S$ is non-compact, then ${\rm Aut}(S)$ might be infinite, but in any case it is a countable group.
In \cite{Greenberg}, Greenberg proved that every finite or countably infinite group can be realized as a group of conformal automorphisms of some Riemann surface of negative Euler characteristic. Moreover, if the group is finite, then the surface can be chosen to be closed (in fact, the closed Riemann surface can also be chosen for the group to coincide with the full group of conformal automorphisms. For infinite-order groups, this fact was proved by Allcock \cite{Allcock}).

A group $G \leq {\rm Aut}(S)$ is said to act {\it purely non-free} (or that it produces a {\it gpnf-action}) if each of its finite-order elements has a fixed point. 
This property has probably been known for a long time, but in the authors' knowledge, J. Gilman was the first to turn attention to it in the context of periodic actions on surfaces in her study of adapted basis for the derived action on the first integral homology group \cite{gilman}. If $G$ has no torsion elements, then every conformal action is gpnf; so this property is interesting only for groups with finite-order elements. In  \cite{BGH1}, gpnf-actions of finite groups on closed Riemann surfaces were studied.

%%%%%%%%
\subsection{Realization results}
In \cite{BGH1}, it was proved that every finite group can be realized as a gpnf-action on a suitable closed Riemann surface of genus at least two.
Our first result states that a similar situation occurs with infinite countable groups. In this case, we need to find non-compact Riemann surfaces with a gpnf-action of the group.

In what follows, a topological surface $X$ is assumed to be connected, Hausdorff, second-countable, and orientable. The classification of surfaces of infinite-type can be found in \cite{Ker,Richards}. The space of ends  ${\rm Ends}(X)$ of a surface $X$ (see \cite{Fre}*{1. Kapitel}) admits a Hausdorff topology homeomorphic to a closed subset of the Cantor set.
This space is called self-similar if, given any pairwise-disjoint clopen subsets $E_{1}, \ldots, E_{n}$
such that ${\rm Ends}(X)= E_{1} \cup \cdots \cup E_{n}$, there exists some $i \in \{1,\ldots,n\}$ and $A \subset E_{i}$ such that $A$ is open
and homeomorphic to ${\rm Ends}(X)$.
In \cite{APV}, the authors proved the following conformal realization result. If $X$ is an infinite-type orientable surface, without planar ends and whose space of ends is self-similar, then every countable group can be realized as a group of conformal automorphisms of a suitable Riemann surface structure on $X$. Their conformal actions might or might not be gpnf-actions. Our first result shows that we can, in fact, choose a gpnf action.

\begin{theo}\label{teo1}
Let $G$ be a finite or an infinite countable group. If $X$ is an infinite-type surface, without planar ends and whose space of ends is self-similar, then there is a Riemann surface structure $S$ on $X$ admitting $G \leq {\rm Aut}(S)$ acting gpnf. 
\end{theo}

An example of a surface $X$ satisfying the hypothesis of Theorem \ref{teo1} is the Loch Ness monster (LNM), the unique (up to homeomorphism) infinite-genus orientable surface with exactly one end.

\begin{coro}
Let $G$ be a finite or infinite countable group. Then there is a Riemann surface structure $S$ on the LNM admitting $G \leq {\rm Aut}(S)$ acting gpnf.
\end{coro}

%%%%%%%%%%%%
\subsection{Upper bounds for gpnf-actions}
Hurwitz's bound asserts that, for $S$ a Riemann surface of genus $g \geq 2$, $|{\rm Aut}(S)|\leq 84(g-1)$. If we restrict to gpnf-actions, the upper bound is much smaller.

\begin{theo}\label{cota}
Let $S$ be a closed Riemann surface of genus $g \geq 2$, and $G \leq {\rm Aut}(S)$ acting gpnf. 
\begin{enumerate}[leftmargin=15pt]
\item If either: (i) $g$ is even or (ii) $g$ is odd and $(g-1)$ has a prime factor different from $2$ and $5$, then $|G| \leq 24(g-1)$.
\item If $S/G$ has no signature of the form $(0;2,4,n)$ and $(g-1)$ has a prime factor $p \notin \{2,3,5,7,11,13,19,23,29,47\}$, then $|G| \leq 8(g-1)$.
\item If $(g-1)$ has a prime factor $p \notin \{2,3,5,7,11,13,17,19,23,29,31,47\}$, then $|G| \leq 9(g-1)$.
\end{enumerate}
\end{theo}

\begin{rema}
Let $g \geq 2$.
The closed Riemann surface $S_{W}: w^{2}=z(z^{2g}-1)$ (called the Wimann curve) has genus $g$ admits the generalized quasi-dihedral group $GQ_{g}=\langle x,y:x^{4g}=y^{2}=1, yxy=x^{2g-1}\rangle$ (of order $8g$) as a group of conformal automorphisms. In \cite{BGH2}, it was observed that, if $g \geq 2$ is even, then the action of $GQ_{g}$ is gpnf, and the quotient signature of $S_{W}/GQ_{g}$ is $(0;2,4,4g)$. If $g \geq 3$ is odd, then that action is not gpnf. But, in this case, the dicyclic group $DC_{g} \leq {\rm Aut}(S)$ (a subgroup of order $4g$) acts gpnf. Moreover, 
there are infinitely many even values of $g$ so that, if $G$ acts gpnf on a closed Riemann surface of such a genus, then $|G| \leq 8g$. 
\end{rema}

%%%%%%%%%%%%
\subsection{Minimal genus results}
Let $G$ be a non-trivial finite group.

%%%%%%%%%%%
\subsubsection{\bf The strong symmetric genus}
The {\it strong symmetric genus} of $G$, denoted by $\sigma^{0}(G)$, is the minimal genus of a closed Riemann surface admitting $G$ as a group of conformal automorphisms. 
As a consequence of Greenberg's result, one can also define the parameter $\sigma_{hyp}^{0}(G)$ as the minimum genus at least two of a closed Riemann surface admitting $G$ as a group of conformal automorphisms.  The case $\sigma^{0}(G)=0$ holds exactly when $G$ is either a cyclic group ${\mathbb Z}_{n}$ (of order $n$), or the dihedral group ${\mathbb D}_{n}$ (of order $2n$), or the alternating group ${\mathcal A}_{4}$ (of order $12$), or the alternating group ${\mathcal A}_{5}$ (of order $60$), or the symmetric group ${\mathfrak S}_{4}$ (of order $24$). The description of those with $\sigma^0(G)=1$ is provided in \cite{GroT}.
For the rest of groups $G$, it holds that $\sigma^{0}(G)=\sigma^{0}_{hyp}(G)$. In \cite{May3}, it was proved that every genus is the strong symmetric genus of some group. The reader may consult, for instance, \cite{MZ7,May3,May10}, for some general results on these parameters.

%%%%%%%%%%%
\subsubsection{\bf The purely non-free genus}
Similarly, as a consequence of the realization results in \cite{BGH1}, one can also define the parameter $\sigma^{pnf}(G)$ as the minimum genus of a closed Riemann surface admitting a gpnf-action by $G$, and the parameter $\sigma_{hyp}^{pnf}(G)$ (which we will call the {\it purely non-free genus} of $G$) as the minimum genus of at least two of a closed Riemann surface admitting a gpnf-action by $G$. Again, the equality $\sigma^{pnf}(G)=0$ holds if and only if $G$ is either cyclic, or a dihedral group, of ${\mathcal A}_{4}$, or ${\mathcal A}_{5}$, or ${\mathfrak S}_{4}$. In the complementary cases, $\sigma^{pnf}(G)=\sigma^{pnf}_{hyp}(G)$ (i.e., there are is no group $G$ with $\sigma^{pnf}(G)=1$).

%%%%%%%%%%
\subsubsection{}
By the definitions, $\sigma^{0}_{hyp}(G) \leq \sigma_{hyp}^{pnf}(G)$. In Section \ref{Sec:ejemplos}, we provide examples for which one has $\sigma^{0}_{hyp}(G) < \sigma_{hyp}^{pnf}(G)$, so the fact that every genus $g \geq 2$ is the strong genus of some group does not imply that the same holds for the purely non-free genus. We will see (see Corollary \ref{DCn}) that, for every integer $n \geq 2$, $\sigma^{pnf}(DC_{n})=\sigma^{pnf}_{hyp}(DC_{n})=n$; thus every genus $g \geq 2$ is the purely non-free genus of some group. An interesting question is to describe the gpnf-spectrum of $G$, i.e., the set ${\rm gpnf}-{\rm spec}(G)$ consisting of those genera $g$ of Riemann surfaces admitting a gpnf-action of $G$. We will not discuss this problem here.
In this paper, we will be mainly interested in the computation of $\sigma_{hyp}^{pnf}(G)$. 

%%%%%%%%%%%
\subsubsection{\bf Some known cases}
Before stating our main results, let us recall some previous results.

{\bf (1) Case of cyclic groups.}
In \cite{H}, Harvey computed $\sigma_{hyp}^{0}({\mathbb Z}_{n})$ for $n \geq 2$ (see Section \ref{Sec:ciclico}). There might be topologically different conformal actions of ${\mathbb Z}_{n}$ on genus $\sigma_{hyp}^{0}({\mathbb Z}_{n})$; for instance, there are two different actions of ${\mathbb Z}_{6}$ on genus two; one is gpnf and the other is not.
 It can be observed that there is always a ${\mathbb Z}_{n}$-action on genus $\sigma_{hyp}^{0}({\mathbb Z}_{n})$ acting gpnf, so $\sigma^{pnf}_{hyp}({\mathbb Z}_{n})=\sigma^{0}_{hyp}({\mathbb Z}_{n})$ \cite{BGH1}.

{\bf (2) Case of abelian groups.}
In \cite{McL}, Maclachlan computed $\sigma_{hyp}^{0}(G)$ for $G$ a non-cyclic abelian group. 
The action of ${\mathbb Z}_{2}^{2}$ on genus two is gpnf, so $\sigma^{pnf}_{hyp}({\mathbb Z}_{2}^{2})=\sigma^{0}_{hyp}({\mathbb Z}_{2}^{2})=2$. In \cite{BGH1}, it was obtained the value of $\sigma^{pnf}_{hyp}(G)$ for $G$ non-cyclic abelian group different from ${\mathbb Z}_{2}^{2}$. 

{\bf (3) Cases: ${\mathcal A}_{4}$, ${\mathcal A}_{5}$ and ${\mathfrak S}_{4}$}
It is known that $\sigma^{0}_{hyp}({\mathcal A}_{4})=\sigma^{pnf}_{hyp}({\mathcal A}_{4})=3$, 
$\sigma^{0}_{hyp}({\mathcal A}_{5})=6<\sigma^{pnf}_{hyp}({\mathcal A}_{5})=30$, and   
$\sigma^{0}_{hyp}({\mathfrak S}_{4})=3<\sigma^{pnf}_{hyp}({\mathfrak S}_{4})=12$.

%%%%%%%%%%%%%%%%%%%
\subsubsection{\bf Case of dihedral groups}
In \cite{McLTesis}, Maclachlan computed $\sigma_{hyp}^{0}(G)$ for $G={\mathbb D}_{n}$ the dihedral group of order $2n$, where $n \geq 3$.
In the following, we describe the value of $\sigma^{pnf}_{hy}({\mathbb D}_{n})$.

\begin{theo}\label{teo2} 
Let $n=p_{1}^{\alpha_{1}} q \geq 3$, be such that $p_{1}$ is the smallest prime divisor of $n$, $\alpha_{1} \geq 1$ and $\gcd(p_{1},q)=1$.
\begin{enumerate}[leftmargin=15pt]
\item $\sigma_{hyp}^{pnf}({\mathbb D}_{3})=\sigma_{hyp}^{0}({\mathbb D}_{3})=2$.

\item If either (i) $\alpha_{1} \geq 2$, or (ii) $\alpha_{1}=q=1$, or (iii) $\alpha_{1}=1$, $q \geq 2$ and $n$ is odd,
then $\sigma_{hyp}^{pnf}({\mathbb D}_{n})=\sigma_{hyp}^{0}({\mathbb D}_{n})=n-n/p_{1}$.

\item If $\alpha_{1}=1$, $q \geq 2$ and $n$ is even, then  $\sigma_{hyp}^{pnf}({\mathbb D}_{n})=n-n/q_{1}$, where $q_{1}$ is the smallest prime divisor of $q$.
\end{enumerate}
\end{theo}

%%%%%%%%%%%%%%
\subsubsection{\bf Computation of $\sigma^{pnf}_{hyp}(G)$}
As we have the value of $\sigma^{pnf}_{hyp}(G)$ for cyclic groups, ${\mathbb Z}_{2}^{2}$, dihedral groups, ${\mathcal A}_{4}$, ${\mathcal A}_{5}$ and ${\mathfrak S}_{4}$. 
The following (whose proof is provided in Section \ref{Sec:prueba0}) corrects the statement in \cite[Proposition 3.3]{BGH1} which, as stated there, is in general false (anyway, for $G$ a non-cyclic abelian group different from ${\mathbb Z}_{2}^{2}$, the result is correct).

\begin{theo}[The value of $\sigma^{pnf}_{hyp}(G)$]\label{teo0}
Let $G$ be a finite group, and $C_{1},\ldots,C_{L_{G}}$ be a collection of pairwise non-conjugate maximal cyclic groups such that $|C_{i}|=m_{i}$.
Let ${\mathcal F}$ be the collection of those tuples $(B_{1},\ldots,B_{L_{G}})$, where each $B_{i}$ is $G$-conjugate to $C_{i}$ (so
$\{B_{1},\ldots,B_{L_{G}}\}$ is a collection of pairwise non-conjugate maximal cyclic subgroups).  For each tuple $B:=(B_{1}=\langle \phi_{1}\rangle,\ldots,B_{L_{G}}=\langle \phi_{L_{G}}\rangle) \in {\mathcal F}$, 
let 
$$m_{B}={\rm Min}\left\{|\phi_{1}^{l_{1}} \cdots \phi_{L_{G}}^{l_{L_{G}}}|: 1 \leq l_{i} \leq |B_{i}|-1, \; \gcd(|B_{i}|,l_{i})=1, \;1 \leq i \leq L_{G}    \right\}.$$
If $m_{G}={\rm Min}\{m_{B}: B \in {\mathcal F}\}$, then
\begin{equation}\label{formula1}
\sigma^{pnf}(G)=1+\frac{|G|}{2}\left(L_{G}-1-\frac{1}{m_{G}}-\sum_{i=1}^{L_{G}}\frac{1}{m_{i}} \right).
\end{equation}
\end{theo}

\begin{rema}
In the case that $G$ is different form either ${\mathbb Z}_{n}$, ${\mathbb Z}_{2}^{2}$, ${\mathbb D}_{n}$, ${\mathcal A}_{4}$, ${\mathcal A}_{5}$ and ${\mathfrak S}_{4}$, then the formula \eqref{formula1} gives the value $\sigma^{pnf}_{hyp}(G)$.

\end{rema}

The above can be implemented in GAP \cite{GAP} (see Section \ref{Sec:ejemplos} for some explicit examples).  
The explicit form of $\sigma^{pnf}_{hyp}(G)$ described in Theorem \ref{teo0} permits computing it for many known families of groups.

\begin{coro}\label{DCn}
If $n \geq 2$ is an integer and $DC_{n}$ is the dicyclic group of order $4n$, then $\sigma^{pnf}(DC_{n})=\sigma^{pnf}_{hyp}(DC_{n})=n$.
Thus, every genus is the purely non-free genus of some group.
\end{coro}

%%%%%%%%%%%%%%
\subsubsection{\bf The metacyclic groups $G^{0}_{s}(n)$}
Let us now  consider the family of metacyclic groups of order $2n$ with presentation of the form
$$G^{0}_{s}(n)=\langle x,y: x^{n}=y^{2}=1, yxy=x^{s}\rangle,$$
where $n \geq 3$ and $2 \leq s \leq n-1$ is such that $s^{2} \equiv 1 \mod(n)$. Note that $G^{0}_{n-1}(n)={\mathbb D}_{n}$; the dihedral group of order $2n$.
The value of $\sigma^{0}_{hyp}(G^{0}_{s}(n))$ has been computed in \cite{Michael}. In \cite{Y2024,Y2026}, the 
value of $\sigma^{pnf}_{hyp}(G^{0}_{s}(n))$ has been computed for the cases $(n,s) \in \{(4m,2m-1), (2^{k-1},2^{k-2}+1) \}$, where $m \geq 2$ and $k \geq 4$ (see Section \ref{Sec:Yerika}).
The next result takes care of some cases when $2 \leq s \leq n-2$. 

\begin{theo}\label{teo3}
Let $n \geq 3$,  and $2 \leq s \leq n-2$ be such that $s^{2} \equiv 1 \mod(n)$. Let $t \in \{1,\ldots,n-1\}$ be such that $s+1=t\gcd(n,s+1)$, and $\gcd(n,t)=1$.
Then
\begin{enumerate}[leftmargin=15pt]
\item The conformal action of $G_{s}^{0}(n)$ on its minimal genus $\sigma^{0}_{hyp}(G_{s}^{0}(n))$ is topologically rigid. Moreover, an algebraic model for such an action is given on the Riemann surface with an algebraic curve model
$$S: W^{n}=U^{\gcd(n,s+1)} (U-1)^{\alpha}(U+1)^{s\alpha},$$
where $\alpha=n-\beta$, and $\beta \in \{1,\ldots,n-1\}$ satisfies that $t\beta \equiv 1 \mod(n)$. In this model, the generators $x$ and $y$ are given as follow ($\omega_{n}=e^{2 \pi i/n}$):
$$x(U,W)=\left(U,\omega_{n} W\right),\;
y(U,W)=\left(-U,Q(U)W^{s} \right),$$
$$Q(U)=\frac{(-1)^{\frac{\gcd(n,s+1)+\alpha(s+1)}{n}}}{U^{\frac{\gcd(n,s+1) (s-1)}{n}} (U+1)^{\frac{\alpha(s^{2}-1)}{n}} }.$$

\item If $n$ is odd, then $\sigma^{pnf}_{hyp}(G_{s}^{0}(n))=\sigma^{0}_{hyp}(G_{s}^{0}(n))$.

\item If $n$ is even, and every maximal cyclic subgroup is conjugate to either $\langle x \rangle$, $\langle y \rangle$, or $\langle xy\rangle$ are a maximal collection of pairwise non-conjugate maximal cyclic groups of $G^{0}_{s}(n)$, then $\sigma^{pnf}_{hyp}(G_{s}^{0}(n))=\sigma^{0}_{hyp}(G_{s}^{0}(n))$.

\end{enumerate}
\end{theo}

\begin{rema}[About the case: $n$ even]
Assume $n$ is even. Since we are assuming the existence of some $2 \leq s \leq n-2$ with $s^{2} \equiv 1 \mod(n)$, necessarily $n \geq 8$. The groups $C_{1}=\langle x \rangle$, $C_{2}=\langle y \rangle$ and $C_{3}=\langle xy\rangle$ are pairwise non-conjugate maximal cyclic subgroups of $G^{0}_{s}(n)$. 
Part (2) of Theorem \ref{teo3} asserts that, if every maximal cyclic subgroup of $G^{0}_{s}(n)$ is conjugate to one of the groups $C_{1}$, $C_{2}$ or $C_{3}$, then $\sigma^{pnf}_{hyp}(G^{0}_{s}(n))=\sigma^{0}_{hyp}(G^{0}_{s}(n))$. In the complementary cases, this equality might fail \cite{Y2024,Y2026}.  But, for these exceptional cases, 
the value of $\sigma^{pnf}_{hyp}(G^{0}_{s}(n))$ can be computed as indicated in Theorem \ref{teo0}. In Section \ref{Sec:ejemplos}, we provide some examples.
\end{rema}

%%%%%%%%%%%%%%%%%%%%%%%%%%
%%%%%%%%%%%%%%%%%%%%%%%%%%
\section{Preliminaries}
\subsection{\bf Maximal cyclic subgroups}
A {\it maximal cyclic subgroup} of a group $G$ is a cyclic group that is not a proper subgroup of another cyclic subgroup of $G$ (see, for instance, \cite{BCLP}). Our interest in maximal cyclic subgroups is the fact that (see in Lemma \ref{lemita1}, Proposition \ref{propo1}) the gpnf property of a conformal action of a group can be described in terms of its maximal cyclic subgroups.
We denote by $L_{G} \geq 1$ (some authors denote this parameter by the symbol $\eta(G)$)  the number of conjugacy classes of its maximal cyclic subgroups. 
It is known that $L_{G}=1$ if and only if $G$ is cyclic. In \cite{Putt}, von Puttkamer proved that if $L_{G}=2$, then $G$ must be solvable with a derived length at most $4$. In \cite{BCLP}, it was observed that $L_{G \times K}=L_{G} L_{K}$ when $\gcd(|G|,|K|)=1$.
(i) If $G$ is abelian, then $L_{G}$ coincides with the number of its maximal cyclic subgroups (for instance, $L_{{\mathbb Z}_{n}}=1$, $L_{{\mathbb Z}_{2}^{2}}=3$).
(ii) If $G$ is a non-cyclic abelian group, different from ${\mathbb Z}_{2}^{2}$, then $L_{G} \geq 3$.
(iii) If $G={\mathbb D}_{n}$, $n \geq 3$, then $L_{{\mathbb D}_{n}}$ is equal to $2$ for $n \geq 3$ odd, and equal to $3$ for $n \geq 4$ even.
(iv) If $G \in \{{\mathcal A}_{3}, {\mathcal A}_{5}, {\mathfrak S}_{4}\}$, then $L_{G}=3$.

%\begin{theo}\label{BCLP}
%Let $G$ be a finite group, and let $G^{-}$ be the subset of elements $x \in G$ with $\langle x \rangle$ not a maximal cyclic subgroup. Then $G/\langle G^{-}\rangle$ 
%is either (i) a $p$-group with exponent $p$, (ii) a Frobenius group whose Frobenius kernel is a $p$-group with exponent $p$ and whose Frobenius complements have order $q$ for some prime $q \neq p$, or (iii) ${\mathcal A}_{5}$.
%\end{theo}

\begin{rema}\label{desi}
Let $G$ be a finite group, different from ${\mathbb Z}_{n}$, ${\mathbb Z}_{2}^{2}$, ${\mathbb D}_{n}$, ${\mathcal A}_{4}$, ${\mathcal A}_{5}$ and ${\mathfrak S}_{4}$ (so $L_{G} \geq 2$).
Let $C_{1},\ldots,C_{L_{G}}$ be a collection of pairwise non-conjugate maximal cyclic groups such that $|C_{i}|=m_{i}$. Then it holds that 
$L_{G}-1>\sum_{i=1}^{L_{G}} \frac{1}{m_{i}}$.
\end{rema}

%%%%%%%%%%%%%%%%%
\subsection{Fuchsian Groups}
Let $S$ be a Riemann surface of negative Euler characteristic and let $G \leq {\rm Aut}(S)$. 
As a consequence of the Riemann uniformization theorem, there is a torsion-free Fuchsian group $\Gamma \cong \pi_{1}(S)$ acting on the hyperbolic plane ${\mathbb H}^{2}$ such that $S$ is conformally equivalent to 
the orbit space ${\mathbb H}^{2}/\Gamma$.
Furthermore, there is another Fuchsian group $\Lambda$, containing $\Gamma$ as a normal subgroup, such that the group $G$ can be identified with $\Lambda/\Gamma$.
Let $\pi: {\mathbb H}^{2} \to X$  be a regular holomorphic covering, with deck group $\Gamma$, and $\theta : \Lambda \to G=\Lambda/\Gamma$ be be the canonical
projection such that, for every $z \in {\mathbb H}$, it holds 
\begin{equation}\label{action}
\theta(\delta)(\pi(z))= \pi (\delta(z)).
\end{equation}

Moreover, the only elements of $\Lambda$ having fixed points in ${\mathbb H}^{2}$ are those of finite order.

%%%%%%%%%%%%%%
\subsubsection{The compact case}
In the particular case that $S$ is a closed Riemann surface of genus $g \geq2$, the group
$\Lambda$ has a presentation
\begin{equation}\label{F presentation}
\langle a_1, b_1, \ldots , a_{\gamma},b_{\gamma}, x_1, \ldots, x_r \,: \, x_1^{k_1}= \cdots =x_r^{k_r}=1= x_1 \cdots x_r [a_1, b_1] \cdots [a_\gamma,b_\gamma]\rangle,
\end{equation}
where the Riemann orbifold $S/G$ has genus $\gamma \geq 0$ and $r \geq 0$ cone points of respective orders $2 \leq k_{1}\leq \cdots\leq k_{r}$. In this case, the 
signature of $\Lambda$ is given by the tuple $(\gamma;k_1, \ldots, k_r)$. The generators $x_i$ represent elliptic elements
of \textit{periods} $k_i$.
Every finite-order element of $\Lambda$ is conjugated to a power of some of the generators $x_i$. 

The hyperbolic area of $\Lambda$ is equal to $2\pi \mu (\Lambda)$, where
\begin{equation}\label{area}
\mu (\Lambda)= \; 2\gamma-2+ \sum_{i=1}^r \left(1 - \dfrac{1}{k_i}\right).
\end{equation}

If $F$ is a finite index $N$ subgroup of $\Lambda$, then  we have the following {\it Hurwitz-Riemann} formula 
${\mu (F)} = N {\mu(\Lambda)}.$

\begin{rema}
Any abstract group with a presentation as in (\ref{F presentation}) can be realized as a Fuchsian group if and only if
(\ref{area}) is positive.
\end{rema}

%%%%%%%%%%%%%%%%
\subsection{Purely non-free actions}\label{Sec:gpnf}
If $S$ is a closed Riemann surface, and $G \leq {\rm Aut}(G)$, then $G$ acts purely non-free (or that $G$ produces a gpnf-action on $S$) if every of its elements of finite order has fixed points.  If $S$ has genus $g \geq 2$, then $G$ is finite. 

A {\it maximal cyclic subgroup} of a group is a cyclic subgroup which is not properly contained in another (necessarily finite) cyclic subgroup of the group. For instance, (i) for a cyclic group of finite order, the only maximal cyclic subgroup is itself, and (ii) if an infinite group has no finite-order non-trivial elements, then it has no maximal cyclic subgroups. 

\begin{rema}
In our definition, a maximal cyclic subgroup might not be a maximal subgroup.
\end{rema}

The importance of maximal cyclic subgroups is given in the following observation.

\begin{lemma}\label{lemita1}
Let $S$ be a Riemann surface and let $G \leq {\rm Aut}(S)$. Then 
\begin{enumerate}[leftmargin=15pt]
\item The group $G$ acts gpnf if and only if each of its maximal cyclic subgroups of finite order acts gpnf. 
\item If $S$ is a closed Riemann surface (so $G$ is finite), and $C_{1},\ldots,C_{L_{G}}$ is a maximal collection of pairwise non-conjugate maximal cyclic groups, then $G$ acts gpnf if and only if each $C_{1},\ldots,C_{L_{G}}$ acts gpnf.
\end{enumerate}
\end{lemma}
\begin{proof}
If $G$ acts gpnf, then every element of finite order has fixed points, in particular, each generator of a maximal cyclic group of finite order. In the other direction, assume that every maximal cyclic group of finite order acts gpnf on $S$. If $\tau \in G$ has finite order, then it belongs to a maximal cyclic group of finite order, so it has fixed points. Last part follows from the fact that every maximal cyclic subgroup of $G$ is conjugate to some of the $C_{i}$'s.
\end{proof}

%%%%%%%%%%%%%%%
\subsection{Signatures of finite gpnf-actions on closed Riemann surfaces}
\begin{prop}\label{propo1}
Let $S$ be a closed Riemann surface of genus $g \geq 2$ and $G\leq {\rm Aut}(S)$ be a finite group acting gpnf. 
Let $C_{1},\ldots,C_{L_{G}}$ be a maximal collection of pairwise disjoint non-conjugate maximal cyclic subgroups.
Then the signature of $S/G$ has the form
\begin{equation}\label{firma}
(\gamma;|C_{1}|,\ldots,|C_{L_{G}}|, k_{1},\ldots,k_{s}),
\end{equation}
where (if $s \geq 1$), each $k_{i} \geq 2$ is the order of some element of $G$.
\end{prop}
\begin{proof}
This follows from Lemma \ref{lemita1}.
\end{proof}

\begin{rema}
Proposition \ref{propo1} states that, to obtain the value of $\sigma^{pnf}_{hyp}(G)$, one needs to minimize 
$\mu(\Lambda)$ for Fuchsian groups $\Lambda$ with signature of types as in \eqref{firma}.
In general, such a minimum occurs for $\gamma=0$, $L_{G} \geq 2$, and $s \in \{0,1\}$ with $L_{G}+s \geq 3$.
\end{rema}

%%%%%%%%%%%%%%%%%
\subsection{Minimal genus $\sigma^{0}_{hyp}({\mathbb Z}_{n})$ actions of ${\mathbb Z}_{n}$ are gpnf}\label{Sec:ciclico}
If $S$ is a Riemann surface and ${\mathbb Z}_{n} \cong G=\langle x \rangle \leq {\rm Aut}(S)$ acts gpnf, then the signature of $S/G$ is of the form $(\gamma;n,k_{1},\ldots,k_{r})$, where $k_{i}\geq 2$ are divisors of $n$. In \cite{H}, Harvey computed the value of $\sigma^{0}_{hyp}({\mathbb Z}_{n})$, for $n \geq 2$. 

\begin{theo}[\cite{H}]
Let $n=p_{1}^{\alpha_{1}}q \geq 2$, where $p_{1}$ is a the smalles prime divisor of $n$, $\alpha_{1} \geq 1$, and $\gcd(p_{1},q)=1$. Then
$$\sigma^{0}_{hyp}({\mathbb Z}_{n})=\left\{ \begin{array}{ll}
{\rm max}\left\{ 2, \left(\frac{p_{1}-1}{2}\right)\left(\frac{n}{2}-1\right)\right\}, & \mbox{$\alpha_{1}=1$ and $q\geq 2$.}\\
{\rm max}\left\{ 2, \left(\frac{p_{1}-1}{2}\right)\frac{n}{p_{1}}\right\}, & \mbox{$\alpha_{1} \geq 2$ or $\alpha_{1}=q=1$.}
\end{array}
\right.
$$
\end{theo}

If $n \in \{2,3,4,6\}$, then $\sigma^{0}_{hyp}({\mathbb Z}_{n})=\sigma^{pnf}_{hyp}({\mathbb Z}_{n})=2$. There are two topologically different actions of ${\mathbb Z}_{2}$ in genus two, one with quotient signature $(0;2,2,2,2,2,2)$ and the other with quotient signature $(1;2,2)$; both being gpnf. 
Similarly, there are two topologically different actions of ${\mathbb Z}_{6}$ on genus two, one with quotient signature $(0;2,2,3,3)$ (so it is not gpnf) and the other with quotient signature $(0;2,6,6)$ (so it is gpnf). 

For the rest of the values of $n$, the minimal actions provide quotient signature $(0;n,n,p_{1})$ or $(0;n,p_{1},\frac{n}{p_{1}})$; so 
these actions are gpnf, in particular, $\sigma^{pnf}_{hyp}({\mathbb Z}_{n})=\sigma^{0}_{hyp}({\mathbb Z}_{n})$. For some of these values of $n$, there are topologically different actions on the minimal genus, but in any of these cases the action is gpnf \cite{BGH1}).

%%%%%%%%%%%%%%%
\subsection{Minimal genus $\sigma^{0}_{hyp}({\mathbb Z}_{2}^{2})$ action of ${\mathbb Z}_{2}^{2}$ is gpnf}
The group $G={\mathbb Z}_{2}^{2}$ admits a conformal action on genus two, so $\sigma^{0}_{hyp}({\mathbb Z}_{2}^{2})=2$. This action is provided on the Riemann surfaces
$$S:W^{2}=(U^{2}-1)(U^{2}-a^{2})(U^{2}-b^{2}),$$
where $a,b \in {\mathbb C}\setminus\{0,\pm1\}$ are such that $a^{2} \neq b^{2}$. The action of $G$ is given by
$x(U,W)=(-U,W)$ and $y(U,W)=(U,-W)$. It can be seen that each $x,y$ and $xy$ acts with fixed points, so $\sigma^{pnf}_{hyp}({\mathbb Z}_{2}^{2})=2$.

%%%%%%%%%%%%%%%%%
\subsection{Construction of some gpnf-actions of finite groups}
As we already know the values of $\sigma^{pnf}_{hyp}(G)$, for $G$ either a cyclic group or ${\mathbb Z}_{2}^{2}$,  we will only need to consider finite groups $G$ different from them.
In particular, the order of $G$ is $N \geq 6$.

\begin{theo}\label{teogpnf}
Let $G$ be a finite group of order $N\geq 6$. Let $C_{1}=\langle \phi_{1}\rangle,\ldots, C_{\hat{M}}=\langle \phi_{\hat{M}}\rangle$ be a collection of (not necessarily different and it might be that some of them are not maximal) cyclic subgroups of $G$ satisfying the following properties:
\begin{enumerate}[leftmargin=15pt]
\item $G=\langle \phi_{1},\ldots,\phi_{\hat{M}}\rangle$, and
\item every maximal cyclic subgroup of $G$ is conjugate to some $C_{i}$.
\end{enumerate}

Let $m_{i}$ be the order of $C_{i}$, and 
$$m={\rm Min}\left\{|\phi_{1}^{l_{1}} \cdots \phi_{\hat{M}}^{l_{\hat{M}}}|: 1 \leq l_{i} \leq m_{i}-1, \; \gcd(m_{i},l_{i})=1, \;1 \leq i \leq \hat{M}    \right\}.$$
Then there is a gpnf-action of $G$ on a closed Riemann surface of genus
\begin{equation}\label{eq1}
g_{G}(C_{1},\ldots,C_{\hat{M}})=1+\frac{N}{2}\left(\hat{M}-1-\frac{1}{m}-\sum_{i=1}^{\hat{M}}\frac{1}{m_{i}} \right).
\end{equation}

The signature of the quotient $S/G$ is $(0;m_{1},\ldots,m_{\hat{M}},m)$ (if $m \geq 2$) or $(0;m_{1},\ldots,m_{\hat{M}})$ (if $m=1$).

If, moreover, 
\begin{equation}\label{eq2}
\hat{M}- \frac{1}{m} - \sum_{i=1}^{\hat{M}} \frac{1}{m_{i}} \geq \frac{N+4}{N},
\end{equation}
then
\begin{equation}\label{eq3}
\sigma^{pnf}_{hyp}(G) \leq  g_{G}(C_{1},\ldots,C_{\hat{M}}).
\end{equation}
\end{theo}

\begin{proof}
May assume $\phi_{i}$ are such that $m=m_{\hat{M}+1}=|\phi_{1}\cdots \phi_{\hat{M}}|$. By our hypothesis, $G=\langle \phi_{1},\ldots,\phi_{\hat{M}}\rangle$. Set $d=-\hat{M}-1-\sum_{i=1}^{\hat{M}+1}\frac{1}{m_{i}}$. Set $\widetilde{X}$ equal to (i) $\widehat{\mathbb C}$ (if $d<0$), or (ii) ${\mathbb C}$ (if $d=0$), or  (iii) 
${\mathbb H}^{2}$ (if $d>0$).

If $m=m_{\hat{M}+1}=1$, let
$$\Lambda=\langle x_{1},\ldots,x_{\hat{M}}: x_{1}^{m_{1}}=\cdots=x_{\hat{M}}^{m_{\hat{M}}}=x_{1}\cdots x_{\hat{M}}=1\rangle \leq {\rm Aut}(\widetilde{X}),$$
and $$\theta:\Lambda \to G: x_{i} \mapsto \phi_{i}; i=1,\ldots,\hat{M}.$$

If $m=m_{\hat{M}+1}\geq 2$, let
$$\Lambda=\langle x_{1},\ldots,x_{\hat{M}+1}: x_{1}^{m_{1}}=\cdots=x_{\hat{M}+1}^{m_{\hat{M}+1}}=x_{1}\cdots x_{\hat{M}+1}=1\rangle \leq {\rm Aut}(\widetilde{X}),$$
and $$\theta:\Lambda \to G: x_{i} \mapsto \phi_{i}, i=1,\ldots,\hat{M}, \; x_{\hat{M}+1} \mapsto (\phi_{1}\cdots\phi_{\hat{M}})^{-1}.$$

In each of the above cases, $\Gamma=\ker(\theta)$ is a torsion-free subgroup of $\Lambda$ and $S=\widetilde{X}/\Gamma$ is a closed Riemann surface admitting $G$ as a group of conformal automorphisms. As a consequence of Lemma \ref{lemita1}, the conformal action of $G$ is gpnf. As a consequence of the Riemann-Hurwitz formula, the genus of $S$ is $g_{G}(C_{1},\ldots, C_{\hat{M}})$. Equation \eqref{eq2} is equivalent to $g_{G}(C_{1},\ldots, C_{\hat{M}}) \geq 2$.
\end{proof}

\begin{rema}
Let $M_{G} \geq L_{G}$ be the number of maximal cyclic subgroups of a finite group $G$. Then: (1) $M_{G}=1$ if and only if $G$ is cyclic; in which case, $\sigma^{pnf}_{hyp}({\mathbb Z}_{n})=\sigma^{0}_{hyp}({\mathbb Z}_{n})$. (2)  If $G$ is not cyclic, then $M_{G} \geq 3$. (3) $M_{G}=3$ if and only if $G={\mathbb Z}_{2}^{2}$, in which case, $\sigma^{pnf}_{hyp}({\mathbb Z}_{2}^{2})=2$. (4) $M_{G} \geq 4$ if and only if $g_{G} \geq 2$. In particular, $\sigma^{pnf}_{hyp}(G) \leq g_{G}$.
\end{rema}

\begin{coro}[\cite{BGH1}]\label{teoabeliano}
Let $G$ be a group, different from a cyclic group or ${\mathbb Z}_{2}^{2}$, so of order $N \geq 6$. Let $C_{1}=\langle \phi_{1}\rangle,\ldots, C_{M_{G}}=\langle \phi_{M_{G}}\rangle$ be all of the maximal cyclic subgroups of $G$. Let $m_{i}$ be the order of $C_{i}$, and  
$m={\rm Min}\left\{|\phi_{1}^{l_{1}} \cdots \phi_{M_{G}}^{l_{M_{G}}}|: 1 \leq l_{i} \leq m_{i}-1, \; \gcd(m_{i},l_{i})=1, \;1 \leq i \leq M_{G}    \right\}.$
Then
\begin{enumerate}[leftmargin=15pt]
\item $$\sigma^{pnf}_{hyp}(G) \leq g_{G}=1+\frac{N}{2}\left(M_{G}-1-\frac{1}{m}-\sum_{i=1}^{M_{G}}\frac{1}{m_{i}} \right).$$

\item If, moreover,  $G$ is abelian, then $\sigma^{pnf}_{hyp}(G) = g_{G}$. 
\end{enumerate}
\end{coro}

The above result (part (2)) was used in \cite{BGH1} to obtain the value of $\sigma^{pnf}_{hyp}(G)$ for some non-cyclic abelian groups. 

\begin{coro}[\cite{BGH1}]
Let $G={\mathbb Z}_{p}^{m}$, where $p$ is a prime integer, and $m \geq 2$ (if $p=2$ we require $m \geq 3$). Then $\sigma_{hyp}^{pnf}(G)={p^{m-1}(p^m-2p-1)}/{2}+1$.
\end{coro}

In \cite{BGH2}, it was observed that some particular (triangular) conformal actions cannot be gpnf.

\begin{theo}[\cite{BGH2}]\label{grupotriangular}
Let $S$ be a closed Riemann surface of genus $g \geq 2$. If $G \leq {\rm Aut}(S)$ is such that $S/G$ has signature $(0;2,3,n)$ for some $n \geq 7$, then the action of $G$ is non-gpnf.
\end{theo}

\begin{rema}[Strategy to obtain Theorem \ref{teo0}]\label{estrategia}
Assume $G$ is a finite group that is different from cyclic, dihedral, ${\mathcal A}_{4}$, ${\mathcal A}_{5}$, ${\mathfrak S}_{4}$; so $L_{G} \geq 2$. 
We also assume that there exists a collection $C_{1}=\langle \phi_{1}\rangle \cong {\mathbb Z}_{m_{1}},\ldots,C_{L_{G}}=\langle \phi_{L_{G}}\rangle \cong{\mathbb Z}_{m_{L_{G}}}$ of pairwise non-conjugate maximal cyclic subgroups of $G$ such that $G=\langle \phi_{1},\ldots,\phi_{L_{G}}\rangle$. 
By Proposition \ref{propo1}, if $G$ acts gpnf on a closed Riemann surface $S$ of genus at least two, then the signature of $S/G$ must be of the form $(\gamma;m_{1},\ldots,m_{L_{G}},k_{1},\ldots,k_{r})$. Let $m \geq 1$ is the order of $\phi_{1}\cdots \phi_{L_{G}}$. 
By Theorem \ref{teogpnf}, there is a gpnf-action of $G$ with $\gamma=0$ and $r=0$ (respectively, $r=1$) if $m=1$ (respectively, $m \geq 2$).
In fact, if we consider 
$$\Delta=\langle x_{1},\ldots,x_{L_{G}+1}:x_{1}^{m_{1}}=\cdots=x_{L_{G}}^{m_{L_{G}}}=x_{L_{G}+1}^{m}=1=x_{1} \cdots x_{L_{G}+1} \rangle,$$
where we delete the generator $x_{L_{G}+1}$ if $r=0$,
then we may consider the surjective homomorphism $\theta:\Delta \to G$ defined by $\theta(x_{j})=\phi_{j}$ (for $1 \leq j \leq L_{G}$), and $\theta(x_{L_{G}+1})=\phi_{1}\cdots \phi_{L_{G}}$. This surjective homomorphism has a torsion-free kernel $\Gamma_{\theta}$. The Riemann surface $S={\mathbb H}^{2}/\Gamma_{\theta}$, of genus $g \geq 2$, admits $G$ as a group of conformal automorphisms acting gpnf, and with quotient signature $(0;m_{1},\ldots,m_{L_{G}},m)$ if $m \geq 2$ or $(0;m_{1},\ldots,m_{L_{G}})$ if $m=1$. So, in order to obtain $\sigma^{pnf}_{hyp}(G)$, we need to minimize the value of $m$.
\end{rema}

%%%%%%%%%%%%%%%%%%%%
%%%%%%%%%%%%%%%%%%%%%%%
\section{The family $G^{0}_{s}(n)$}
In this paper, we consider the family of metacyclic groups of order $2n \geq 6$ with presentation of the form
$$G^{0}_{s}(n)=\langle x,y: x^{n}=y^{2}=1, yxy=x^{s}\rangle,$$
where $n \geq 3$ and $2 \leq s \leq n-1$ is such that $s^{2} \equiv 1 \mod(n)$. Note that $G^{0}_{n-1}(n)={\mathbb D}_{n}$ (the dihedral group of order $2n$). 

%%%%%%%%%%%%%%%
\subsection{Elements of $G^{0}_{s}(n)$}
Every element of $G^{0}_{s}(n)$ is of the form $x^{l}$ or $x^{l}y$, where $l \in \{0,1,\ldots,n-1\}$. 
The order of $x^{l}y$, where $l\in\{0,1,\ldots,n-1\}$, is even and given by 
$$|x^{l}y|=\frac{2n}{\gcd(n,l(s+1))}.$$

Also, for every $r \in \{1,\ldots,n-1\}$, $k \geq 1$, and $l \in \{0,1,\ldots,n-1\}$, 
$$x^{-r} (x^{l}y) x^{r}=x^{l+r(s-1)}y, \quad 
(x^{l}y)^{2k+1}=x^{l(1+k(s+1))}y.$$

Let us consider the following cyclic subgroups of $G^{0}_{s}(n)$:
$$C_{1}=\langle x \rangle, \; C_{2}=\langle y \rangle, \; C_{3}=\langle xy \rangle, \ldots, C_{n+1}=\langle x^{n-1}y\rangle.$$

Note that every cyclic subgroup of $G^{0}_{s}(n)$ is a subgroup of one of the above ones. So, the maximal cyclic subgroups are included in the above collection.

\begin{lemma}\label{lemita3}
$C_{1}, \; C_{2},$ and $C_{3}$
are maximal cyclic subgroups of $G^{0}_{s}(n)$.
\end{lemma}
\begin{proof}
Since $C_{1} \cong {\mathbb Z}_{n}$, $|G^{0}_{s}(n)|=2n$, and $G^{0}_{s}(n)$ is not cyclic, it follows that $C_{1}$ is maximal cyclic subgroup.

Assume $C_{2} \leq \langle x^{l}y\rangle$, for some $l \in \{2,\ldots,n-1\}$. Then, there is some $k$ such that $xy=(x^{l}y)^{2k+1}=x^{l(k(1+s)+1)}y$. So, $l(k(1+s)+1)\equiv 1 \mod(n)$. By multiplying by $1-s \nequiv 0 \mod(n)$, one gets that $l(1-s) \equiv 0 \mod(n)$. Now, since $(s-1)(s+1)=s^{2}-1 \equiv 0 \mod(n)$, it follows that $\gcd(n,l(s+1))=n$. 
This means that $\langle x^{l}y\rangle \cong {\mathbb Z}_{2}$. In particular, as $y \in \langle x^{l}y\rangle$, must have $l=0$, a contradiction.

Assume $C_{3} \leq \langle x^{l}y\rangle$, for some $l \in \{1,\ldots,n-1\}$. Then, there is some $k$ such that $y=(x^{l}y)^{2k+1}=x^{l(k(1+s)+1)}y$. So, $l(k(1+s)+1)\equiv 0 \mod(n)$. By multiplying by $1-s \nequiv 0 \mod(n)$, one gets that $l(1-s) \equiv 1 \mod(n)$, in particular, that $\gcd(n,l)=1$. So $x^{l}y$ will have order $2n$, a contradiction as $G^{0}_{s}(n$ is not cyclic.
\end{proof}

%%%%%%%%%%%%%%%%
\subsection{The strong symmetric genus of $G^{0}_{s}(n)$}
The computation of the strong symmetric group of $G^{0}_{s}(n)$ was done in \cite{Michael} (the particular case $s=n-1$, i.e., for the dihedral group, this was obtained in \cite{McLTesis}).

\begin{theo}[\cite{Michael}]
Let $n=p_{1}^{\alpha_{1}}q \geq 3$, where $p_{1}$ is the smallest prime divisor of $n$, $\alpha_{1} \geq 1$, $\gcd(p_{1},q)=1$, and $2 \leq s \leq n-1$ with $s^{2}\equiv 1 \mod(n)$. Then
$$\sigma^{0}_{hy}(G^{0}_{s}(n))=\left\{\begin{array}{ll}
\dfrac{n-\gcd(n,s+1)}{2}, & s \neq n-1.\\
n+1-p_{1}-\dfrac{n}{p_{1}}, & s=n-1, \; \alpha_{1}=1, \; q \geq 2.\\
n-\dfrac{n}{p_{1}}, & s=n-1, \; \mbox{and either $\alpha_{1}\geq 2$, or $\alpha_{1}=q=1$}.
\end{array}
\right.
$$

Moreover, if $S$ is a closed Riemann surface admitting $G^{0}_{s}(n)$ as a group of conformal automorphisms, and $S$ has the minimal genus $\sigma^{0}_{hy}(G^{0}_{s}(n))$, then $S/G^{0}_{s}(n)$ has signature:
$$\left\{\begin{array}{ll}
\left(0;2,\dfrac{2n}{\gcd(n,s+1)},n\right), & s\neq n-1.\\
\left(0;2,2,p_{1},\frac{n}{p_{1}}\right), & s=n-1, \; \alpha_{1}=1, \; q \geq 2.\\
\left(0;2,2,p_{1},n\right), & s=n-1, \; \mbox{and either $\alpha_{1}\geq 2$, or $q=1$}.
\end{array}
\right\}.
$$
\end{theo}

\begin{rema}
Let us observe that the signature $\left(0;2,2,p_{1},\frac{n}{p_{1}}\right)$, produced in the case that $s=n-1$, $\alpha_{1}=1$ and $q \geq 2$, cannot be produced by a gpnf-action of $G^{0}_{n-1}(n)={\mathbb D}_{n}$.
As we will see later, the signature $\left(0;2,2,p_{1},n\right)$, produced in the case that $s=n-1$, and either (i) $\alpha_{1}\geq 2$, or (ii) $q=1$, is produced by a gpnf-action of $G^{0}_{n-1}(n)={\mathbb D}_{n}$.
\end{rema}

%%%%%%%%%%%%%%
\subsection{Some examples for $\sigma^{pnf}_{hyp}(G^{0}_{s}(n))$}\label{Sec:Yerika}
As a consequence of Theorem \ref{teo0}, for $s \neq n-1$, one can compute explicitly the value of $\sigma^{pnf}_{hyp}(G^{0}_{s}(n))$.
The computation of $\sigma^{pnf}_{hyp}(G^{0}_{s}(n))$ was previously obtained for some cases (quasi-abelian groups and quasi-dihedral groups).

\begin{theo}[\cite{Y2024,Y2026}]
$$\sigma^{0}_{hyp}(G^{0}_{2^{m-2}+1}(2^{m-1}))=2^{m-2}-1 < \sigma^{pnf}_{hyp}(G^{0}_{2^{m-2}+1}(2^{m-1}))=\left\{\begin{array}{ll}
2^{m-1}(m-3)+1, & m \geq 5\; \mbox{odd}.\\
2^{m-2}(2m-3)+1, & m \geq 4\; \mbox{even}.
\end{array}
\right.
$$
$$\sigma^{0}_{hyp}(G^{0}_{2m-1}(4m))=m \leq \sigma^{pnf}_{hyp}(G^{0}_{2m-1}(4m))=\left\{\begin{array}{ll}
m, & m \geq 2\; \mbox{even}.\\
3m, & m \geq 3\; \mbox{odd}.
\end{array}
\right.
$$
\end{theo}

%%%%%%%%%%%%%%%%
\subsection{An auxiliary lemma}
Let us consider the case when $2 \leq s \leq n-2$. As $s^2 \equiv 1 \mod(n)$, necessarily $n \geq 8$ and $n$ is not a prime integer. Some examples are:

If $n=8$, then $s \in \{3,5\}$.

If $n=9,10,14$, then there is no solution for $s$ as above.

If $n=12$, then $s\in\{5,7\}$.

\begin{lemma}\label{lemita2}
Let $n=2^{r}q \geq 3$, where $r \geq 0$ and $q\geq 1$ is odd, and $2 \leq s \leq n-2\}$ such that $s^{2} \equiv 1 \mod(n)$.
\begin{enumerate}[leftmargin=15pt]
\item If either (i) $r\in\{0,1,2\}$, or (ii) $s \neq 1+2^{r-1}q$, then the elements of order $n$ in $G^{0}_{s}(n)$ are those of the form $x^{u}$, where $u \in \{1,\ldots,n-1\}$ and $\gcd(n,u)=1$.

In particular, if $\phi \in {\rm Aut}(G^{0}_{s})$, then
$$\phi(x)=x^{u}, \; \phi(y)=x^{v}y,$$
where $u \in \{1,\ldots,n-1\}$ and $\gcd(n,u)=1$, $v \in \{0,1,\ldots,n-1\}$ and $\gcd(n,v(s+1))=n$.

\item If $r \geq 3$ and $s=1+2^{r-1}q$, then the elements of order $n$ in $G^{0}_{s}(n)$ are: (i) $x^{u}$, where $u \in \{1,\ldots,n-1\}$ and $\gcd(n,u)=1$, or 
(ii) $x^{l}y$, where $l\in\{1,\ldots,n-1\}$ and $\gcd(n,l)=1$.

Moreover, apart for those elements of ${\rm Aut}(G^{0}_{s}(n))$ as described in the above case, there are also elements $\psi \in {\rm Aut}(G^{0}_{s}(n))$ of the form
$$\psi(x)=x^{l}y, \; \phi(y)=x^{v}y,$$
where $l \in \{1,\ldots,n-1\}$ and $\gcd(n,l)=1$, $v \in \{0,n/2\}$.

\end{enumerate}
\end{lemma}
\begin{proof}
Assume there is an element of the form $x^{l}y \in G^{0}_{s}(n)$ of order $n$. Since $|x^{l}y|=\frac{2n}{\gcd(n,l(s+1))}$, we must have that
\begin{equation}\label{*}
\gcd(n,l(s+1))=2.
\end{equation}
This in particular asserts that $r \geq 1$. Since $s^{2} \equiv 1 \mod(n)$, it follows that $n$ divides $(s-1)(s+1)$. Since $n$ is even, so $s$ must be odd, equality \eqref{*} asserts that (i) $\gcd(n,l)=1$, and (ii) $\gcd(n,s+1)=2$. In particular, $s+1=2\alpha$, where $\gcd(n,\alpha)=1$.

As $n=2^{r}q$ divides $(s-1)(s+1)$, the above asserts that $s-1=2^{r-1} q \;\beta$ is even, for some $\beta \geq 1$. As $2\alpha=s+1$, we obtain that
$\alpha=1+2^{r-2} q \; \beta$.

Since $s \leq n-2=2^{r}q-2$, it follows that $s+1 \leq 2^{r}q-1$. In particular,
$$1+2^{r-2}q \;\beta=\alpha=(s+1)/2= \leq \frac{2^{r}q-1}{2},$$
so
\begin{equation}\label{**}
2+2^{r-1}q \; \beta \leq 2^{r}q-2,
\end{equation}
and
\begin{equation}\label{***}
2^{r-2}q \; \beta \leq 2^{r-1}q-2,
\end{equation}
from which we obtain that 
\begin{equation}\label{****}
2^{r-3}q \; \beta \leq 2^{r-2}q-1.
\end{equation}

If $r=1$, then $\beta \geq 2$ is even and, by \eqref{**}, $2+q\; \beta \leq 2q-2$, implying that $q(\beta-2) \leq -4$, a contradiction.

If $r=2$, then $\beta \leq 2$ is even and, by \eqref{***}, $q \; \beta \leq 2q-2<22$, from which we obtain that $\beta<2$, a contradiction.

If $r \geq 3$, then, by \eqref{****}, $\beta=1$ and $\alpha=1+2^{r-2}q$,  so, $s=1+2^{r-1}q$.

The above takes care of parts (1) and (2), with the exception of the form of the automorphisms.

The map $\phi$ as described in the hypothesis can be seen to be an isomorphism.

Let us now consider the case when $n=2^{r}q$, where $r \geq 3$ and $q \geq 1$ is odd. As seen above, if $x^{l}y$ has order $n$, then
\begin{enumerate}[leftmargin=15pt]
\item $l \geq 1$, $\gcd(n,l)=1$;
\item $s=1+2^{r-1}q$.
\end{enumerate}

So, $\gcd(n,l(s+1))=\gcd(n,s+1)=2$.

Let us consider the map $\psi$ as described in part (2) of the Lemma. One may see, by the definition of $\psi$, that $\psi(x)$ has order $n$, and that $\psi(y)$ has order two.
Also, since $2v \equiv 0 \mod(n)$, 
$$\psi(y)\psi(x)\psi(y)=(x^{v}y)(x^{l}y)(x^{v}y)=x^{2v+ls}y=x^{ls}y.$$ 
On the other hand,
$$\psi(x)^{s}=(x^{l}y)^{s}=(x^{l}y)^{1+2^{r-1}q}=((x^{l}y)^{2})^{2^{r-2}q} (x^{l}y)=(x^{l(s+1)})^{2^{r-2}q} (x^{l}y)=$$
$$=x^{l(s+1)(s-1)/2 +1}y =x^{l(s^{2}+1)/2} y.$$

The equality $x^{ls}=x^{l(s^{2}+1)/2} y$ is equivalent to have
$l((s^{2}+1)/2-s) \equiv 0 \mod(n)$, and since $\gcd(n,l)=1$, this is equivalent to $(s^{2}+1)/2-s \equiv 0 \mod(n)$, equivalently, $(s-1)^{2}/2 \equiv 0 \mod(n)$. As $s=1+2^{r-1}q$, the previous is equivalent to$2^{2r-3}q^{2} \equiv 0 \mod(n)$, which is true for $r \geq 3$. So, all the above ensures that $\psi$ is a homomorphism. It can be seen that $\psi$ is injective, so it is an automorphism as required.
\end{proof}

%%%%%%%%%%%%%%%
\subsection{Three maximal cyclic subgroups of $G^{0}_{s}(n)$ and $\sigma^{pnf}_{hyp}(G^{0}_{s}(n))$}
Let us consider the maximal cyclic subgroups $C_{1}=\langle x \rangle \cong {\mathbb Z}_{n}$, $C_{2}=\langle y \rangle \cong {\mathbb Z}_{2}$, and $C_{3}=\langle xy \rangle \cong {\mathbb Z}_{2/\gcd(n,s+1)}$ (Lemma \ref{lemita3}). As $n \geq 3$, $C_{1}$ and $C_{2}$ are non-conjugate. If, moreover, $2 \leq s \leq n-2$,  then these three are pairwise non-conjugate.

As a consequence of Lemma \ref{lemita1}, we observe the following (which provides a proof of part (3) of Theorem \ref{teo3}).

\begin{prop}\label{casosimple}
Assume $n \geq 3$ and $2 \leq s \leq n-1$ are integers such that $s^{2} \equiv 1 \mod(n)$. If every maximal cyclic subgroup is conjugate to either $C_{1}$, $C_{2}$, or $C_{3}$, then $\sigma^{pnf}_{hyp}(G^{0}_{s}(n))=\sigma^{0}_{hyp}(G^{0}_{s}(n))$.
\end{prop}
\begin{proof}
Let us consider the Fuchsian group
$$\Delta=\langle x_{1},x_{2},x_{3}: x_{1}^{n}=x_{2}^{2}=x_{3}^{\frac{2n}{\gcd(n,s+1)}}=x_{1}x_{2}x_{3}=1\rangle,$$
and the surjective homomorphism (with a torsion-free kernel)
$$\theta:\Delta \to G^{0}_{s}(n): 
\theta(x_{1})=x, \; \theta(x_{2})=y, \; \theta(x_{3})=x^{-s}y.$$

The conformal action  of $G^{0}_{s}(n)$, induced by $\theta$, is produced on a closed Riemann surface of genus $\sigma^{0}_{hyp}(G^{0}_{s}(n))$. As each of the maximal cyclic subgroups $C_{1}$, $C_{2}$ and $C_{3}$ acts with fixed points (by our hypothesis) (and by Lemma \ref{lemita1}), it follows that such a conformal action is gpnf.
\end{proof}

%%%%%%%%%%%%%%%%%%%%%%%
%%%%%%%%%%%%%%%%%%%%%%%
\section{Proof of Theorem \ref{teo1}}
Let $G$ be a countable group. By Greenberg's realization theorem, we may assume that $G$ acts conformally on some Riemann surface $S$. Moreover, if $G$ is finite, then $S$ can be assumed to be a closed Riemann surface of genus at least two. 
If we fix an infinite-type orientable surface $X$ without planar ends and with a self-similar space of ends, then by the results in \cite{APV}, we may assume $S$ to be homeomorphic to $X$.

Let us consider the Riemann orbifold $R=S/G$ and let $\pi:S \to R$ be a Galois (branched) covering with deck group $G$. If the conformal action of $G$ is gpnf, then we are done. So, let us assume this is not the case. 
Let $C_{i}$, where $1 \leq i \leq m$ and $m$ is either an integer or infinity, be a maximal collection of non-conjugate maximal cyclic subgroups of $G$ that act freely on $S$. If $m=\infty$ and $R$ is compact, then we delete a non-branch point from $R$ and its $\pi$-preimages from $S$ (so, for $m=\infty$, we may assume $R$ to be non-compact). 

Let us consider a collection $\{D_{i}\}_{1 \leq i \leq m}$ of pairwise disjoint closed discs on $R$, each one disjoint from the branch locus of $\pi$
(if $m=\infty$, then necessarily these discs approach some of the ends of $R$) .

For each $i$, we proceed to delete the interior of $D_{i}$ and the interior of all the preimages discs $D_{i,1}, D_{i,2}, \ldots $ of $\pi^{-1}(D_{i})$; we denote by $S^{*}$ the resulting surface. We reorder the borders of all these discs $D_{i,j}$ into a disjoint collection $F_{i,1},F_{i,2}\ldots$, each one formed of $n_{i}=|C_{i}|$ discs that are in the same orbit of some of the $G$-conjugates of $C_{i}$.

Next, we may consider a conformal action of $H_{i}={\mathbb Z}_{n_{i}}$ on the Riemann sphere, and again, we delete the interior of the orbit of a small disc (small enough so its $H_{i}$-orbits are pairwise disjoint) to obtain a genus zero bordered surface $T_{i}$ with $n_{i}$ boundaries that are cyclically permuted by $H_{i}$. Then, for each $F_{i,j}$,
we proceed to glue one copy of $T_{i}$ to $S^{*}$, by gluing the $n_{i}$ borders of $T_{i}$ to the $n$ borders of $S^{*}$ belonging to $F_{i,j}$. We do this gluing so that the cyclic action of $H_{i}$ on $T_{i}$ matches the cyclic action of (the conjugates of) $C_{i}$ on each $F_{i,j}$. 

In the end, we obtain a new orientable connected surface $S_{1}$ for which the action of $G$ has the property that each $C_{i}$ acts with fixed points (if the surface $S$ has infinite-type without planar ends, then $S_{1}$ is topologically equivalent to $S$). 

The quotient surface $R_{1}=S_{1}/G$ is obtained from $R$ by adding two branch points, each of order $n_{i}$, to the interior of the disc $D_{i}$. Now, take any Riemann surface structure on $R_{1}$ and lift it to $S_{1}$ to ensure that $G$ acts as a group of conformal automorphisms. The elements of $G$, acting with fixed points originally on $S$, still act with fixed points on $S_{1}$. But now, we have that each $C_{i}$ also acts with fixed points.
The desired result follows now from Lemma \ref{lemita1}.
\qed

%%%%%%%%%%%%%%%%%%%%%%
%%%%%%%%%%%%%%%%%%%%%%
\section{Proof of Theorem \ref{cota}}
Let $S$ be a closed Riemann surface of genus $g \geq 2$, and let $G \leq {\rm Aut}(S)$ be acting gpnf. Le us assume $S/G$ has signature $(\gamma;m_{1},\ldots,m_{r})$. The gpnf-property asserts that every (non-trivial) element of $G$ has a fixed point, so the prime divisors of $|G|$ must belong to the collection of prime divisors of the numbers $m_{1},\ldots,m_{r}$.

\subsection{Case $|G| \geq 24(g-1)$}
As a consequence of the Riemann-Hurwitz formula, the signature of $S/G$ is one of the following \cite{Singerman}:
$$(0;2,3,7), (0;2,3,8), (0;2,4,5), (0;2,3,9), (0;2,3,10), $$ $$(0;2,3,11), (0;2,3,12), (0;2,4,6), (0; 3,3,4).$$

As a consequence of Theorem \ref{grupotriangular}, the only possibilities for $G$ to act purely non-free are for the signatures:
$(0;2,4,5),  (0;2,4,6), (0; 3,3,4).$ For the signature $(0;2,4,5)$: $|G|=40(g-1)$, for $(0;2,4,6)$: $|G|=24(g-1)$, and for $(0;3,3,4)$: $|G|=24(g-1)$.

(1) Let us consider the signature $(0;2,4,5)$. In this case, the gpnf-property ensures that the non-trivial elements of $G$ can only have orders $2, 4$ or $5$. In particular, $|G|=40(g-1)=2^{\alpha}5^{\beta}$, where $\alpha \geq 3$ and $\beta \geq 1$.  So, $(g-1)$ cannot have a prime factor different from $2$ or $5$. Assume $g$ is even; so $\alpha=3$.
As the number of $5$-Sylow subgroups is congruent to $1$ modulo $5$, and divides $|G|$, this number is $1$. Let $K$ be the unique $5$-Sylow subgroup of $G$; so $K \lhd G$. By the property of the orders of the elements of $G$, we have $K \cong {\mathbb Z}_{5}^{\beta}$. The quotient orbifold $S/K$ has a signature of the form $(\gamma;5,\stackrel{n}{\ldots},5)$, where $n \geq 1$. The quotient group $G/K$, of order $2^{\alpha}$, acts gpnf on the underlying Riemann surface $R$ of $S/K$. The quotient of $R/(G/K)$ will have signature $(0;2,4)$, a contradiction.

(2) Let us now consider the signature $(0;3,3,4)$. In this case, the gpnf-property ensures that the non-trivial elements of $G$ can only have orders $2, 3$ or $4$. In particular, $|G|=24(g-1)=2^{\alpha}3^{\beta}$, where $\alpha \geq 3$ and $\beta \geq 1$. So, necessarily $(g-1)$ cannot have a prime factor different from $2$ or $3$.

(3) Finally, let us now consider the signature $(0;2,4,6)$. In this case, the gpnf-property ensures that the non-trivial elements of $G$ can only have orders $2, 3, 4$ or $6$. In particular, $|G|=24(g-1)=2^{\alpha}3^{\beta}$, where $\alpha \geq 3$ and $\beta \geq 1$. So, necessarily $(g-1)$ cannot have a prime factor different from $2$ or $3$.

\subsection{Case $8(g-1)<|G|<24(g-1)$}
Again, the possible signatures for $S/G$ are given in \cite{Singerman}. We now discuss each one.

(1) The signature $(0;2,2,2,3)$ is produced for $|G|=12(g-1)$ and, by the gpnf-property $g-1=2^{\alpha}3^{\beta}$.

(2) For signature $(0;2,4,n)$, where $7 \leq n <36$, we have $8(g-1)<|G|=\left(\frac{8n}{n-4}\right)(g-1)\leq 9(g-1)$. 
The signature $(0;2,4,7)$ is produced for $|G|=\frac{56}{9}(g-1)$, so $9$ divides $(g-1)$. By the gpnf-property $g-1=2^{\alpha}3^{\beta}7^{\gamma}$.
The signature $(0;2,4,8)$ is produced for $|G|=16(g-1)$, and by the gpnf-property $g-1=2^{\alpha}$.
The signature $(0;2,4,9)$ is produced for $|G|=\frac{72}{5}(g-1)$, so $5$ divides $(g-1)$. By the gpnf-property $g-1=2^{\alpha}3^{\beta}5^{\gamma}$.
The signature $(0;2,4,10)$ is produced for $|G|=\frac{40}{3}(g-1)$, so $3$ divides $(g-1)$. By the gpnf-property $g-1=2^{\alpha}3^{\beta}5^{\gamma}$.
The signature $(0;2,4,11)$ is produced for $|G|=\frac{88}{7}(g-1)$, so $7$ divides $(g-1)$. By the gpnf-property $g-1=2^{\alpha}7^{\beta}11^{\gamma}$.
The signature $(0;2,4,12)$ is produced for $|G|=12(g-1)$. By the gpnf-property $g-1=2^{\alpha}3^{\beta}$.
The signature $(0;2,4,13)$ is produced for $|G|=\frac{104}{9}(g-1)$, so $9$ divides $(g-1)$. By the gpnf-property $g-1=2^{\alpha}3^{\beta}13^{\gamma}$.
The signature $(0;2,4,14)$ is produced for $|G|=\frac{56}{5}(g-1)$, so $5$ divides $(g-1)$. By the gpnf-property $g-1=2^{\alpha}5^{\beta}7^{\gamma}$.
The signature $(0;2,4,15)$ is produced for $|G|=\frac{120}{11}(g-1)$, so $11$ divides $(g-1)$. By the gpnf-property $g-1=2^{\alpha}3^{\beta}5^{\gamma}11^{\delta}$.
The signature $(0;2,4,16)$ is produced for $|G|=\frac{32}{3}(g-1)$, so $3$ divides $(g-1)$. By the gpnf-property $g-1=2^{\alpha}3^{\beta}$.
The signature $(0;2,4,17)$ is produced for $|G|=\frac{136}{13}(g-1)$, so $13$ divides $(g-1)$. By the gpnf-property $g-1=2^{\alpha}13^{\beta}17^{\gamma}$.
The signature $(0;2,4,18)$ is produced for $|G|=\frac{72}{7}(g-1)$, so $7$ divides $(g-1)$. By the gpnf-property $g-1=2^{\alpha}3^{\beta}7^{\gamma}$.
The signature $(0;2,4,19)$ is produced for $|G|=\frac{152}{15}(g-1)$, so $15$ divides $(g-1)$. By the gpnf-property $g-1=2^{\alpha}3^{\beta}5^{\gamma}19^{\delta}$.
The signature $(0;2,4,20)$ is produced for $|G|=10(g-1)$. By the gpnf-property $g-1=2^{\alpha}5^{\beta}$.
The signature $(0;2,4,21)$ is produced for $|G|=\frac{168}{17}(g-1)$, so $17$ divides $(g-1)$. By the gpnf-property $g-1=2^{\alpha}3^{\beta}7^{\gamma}17^{\delta}$.
The signature $(0;2,4,22)$ is produced for $|G|=\frac{88}{9}(g-1)$, so $9$ divides $(g-1)$. By the gpnf-property $g-1=2^{\alpha}3^{\beta}11^{\gamma}$.
The signature $(0;2,4,23)$ is produced for $|G|=\frac{184}{19}(g-1)$, so $19$ divides $(g-1)$. By the gpnf-property $g-1=2^{\alpha}19^{\beta}23^{\gamma}$.
The signature $(0;2,4,24)$ is produced for $|G|=\frac{48}{5}(g-1)$, so $5$ divides $(g-1)$. By the gpnf-property $g-1=2^{\alpha}3^{\beta}5^{\gamma}$.
The signature $(0;2,4,25)$ is produced for $|G|=\frac{200}{21}(g-1)$, so $21$ divides $(g-1)$. By the gpnf-property $g-1=2^{\alpha}3^{\beta}5^{\gamma}7^{\delta}$.
The signature $(0;2,4,26)$ is produced for $|G|=\frac{104}{11}(g-1)$, so $11$ divides $(g-1)$. By the gpnf-property $g-1=2^{\alpha}11^{\beta}13^{\gamma}$.
The signature $(0;2,4,27)$ is produced for $|G|=\frac{216}{23}(g-1)$, so $23$ divides $(g-1)$. By the gpnf-property $g-1=2^{\alpha}3^{\beta}23^{\gamma}$.
The signature $(0;2,4,28)$ is produced for $|G|=\frac{28}{3}(g-1)$, so $3$ divides $(g-1)$. By the gpnf-property $g-1=2^{\alpha}3^{\beta}7^{\gamma}$.
The signature $(0;2,4,29)$ is produced for $|G|=\frac{232}{25}(g-1)$, so $5$ divides $(g-1)$. By the gpnf-property $g-1=2^{\alpha}5^{\beta}29^{\gamma}$.
The signature $(0;2,4,30)$ is produced for $|G|=\frac{120}{13}(g-1)$, so $13$ divides $(g-1)$. By the gpnf-property $g-1=2^{\alpha}3^{\beta}5^{\gamma}13^{\delta}$.
The signature $(0;2,4,31)$ is produced for $|G|=\frac{248}{27}(g-1)$, so $27$ divides $(g-1)$. By the gpnf-property $g-1=2^{\alpha}3^{\beta}31^{\gamma}$.
The signature $(0;2,4,32)$ is produced for $|G|=\frac{64}{7}(g-1)$, so $7$ divides $(g-1)$. By the gpnf-property $g-1=2^{\alpha}7^{\beta}$.
The signature $(0;2,4,33)$ is produced for $|G|=\frac{264}{29}(g-1)$, so $29$ divides $(g-1)$. By the gpnf-property $g-1=2^{\alpha}3^{\beta}11^{\gamma}29^{\delta}$.
The signature $(0;2,4,34)$ is produced for $|G|=\frac{136}{15}(g-1)$, so $15$ divides $(g-1)$. By the gpnf-property $g-1=2^{\alpha}3^{\beta}5^{\gamma}$.
The signature $(0;2,4,35)$ is produced for $|G|=\frac{280}{31}(g-1)$, so $31$ divides $(g-1)$. By the gpnf-property $g-1=2^{\alpha}5^{\beta}7^{\gamma}31^{\delta}$.

(3) For signature of the form $(0;2,5,n)$, where $5 \leq n < 20$, one has that $8(g-1) <|G|=\left(\frac{20n}{3n-10}\right)(g-1)<24(g-1)$.

The signature $(0;2,5,5)$ is produced for $|G|=20(g-1)$. By the gpnf-property $g-1=2^{\alpha}5^{\beta}$.
The signature $(0;2,5,6)$ is produced for $|G|=15(g-1)$. By the gpnf-property $g-1=2^{\alpha}3^{\beta}5^{\gamma}$.
The signature $(0;2,5,7)$ is produced for $|G|=\frac{140}{11}(g-1)$, so $11$ divides $(g-1)$. By the gpnf-property $g-1=2^{\alpha}5^{\beta}7^{\gamma}11^{\delta}$.
The signature $(0;2,5,8)$ is produced for $|G|=\frac{80}{7}(g-1)$, so $7$ divides $(g-1)$. By the gpnf-property $g-1=2^{\alpha}5^{\beta}7^{\gamma}$.
The signature $(0;2,5,9)$ is produced for $|G|=\frac{180}{7}(g-1)$, so $7$ divides $(g-1)$. By the gpnf-property $g-1=2^{\alpha}3^{\beta}5^{\gamma}7^{\delta}$.
The signature $(0;2,5,10)$ is produced for $|G|=10g-1)$. By the gpnf-property $g-1=2^{\alpha}5^{\beta}$.
The signature $(0;2,5,11)$ is produced for $|G|=\frac{220}{23}(g-1)$, so $23$ divides $(g-1)$. By the gpnf-property $g-1=2^{\alpha}5^{\beta}11^{\gamma}23^{\delta}$.
The signature $(0;2,5,12)$ is produced for $|G|=\frac{120}{13}(g-1)$, so $13$ divides $(g-1)$. By the gpnf-property $g-1=2^{\alpha}3^{\beta}5^{\gamma}13^{\delta}$.
The signature $(0;2,5,13)$ is produced for $|G|=\frac{260}{29}(g-1)$, so $29$ divides $(g-1)$. By the gpnf-property $g-1=2^{\alpha}5^{\beta}13^{\gamma}29^{\delta}$.
The signature $(0;2,5,14)$ is produced for $|G|=\frac{35}{4}(g-1)$, so $4$ divides $(g-1)$. By the gpnf-property $g-1=2^{\alpha}5^{\beta}7^{\gamma}$.
The signature $(0;2,5,15)$ is produced for $|G|=\frac{60}{7}(g-1)$, so $7$ divides $(g-1)$. By the gpnf-property $g-1=2^{\alpha}3^{\beta}5^{\gamma}7^{\delta}$.
The signature $(0;2,5,16)$ is produced for $|G|=\frac{160}{19}(g-1)$, so $19$ divides $(g-1)$. By the gpnf-property $g-1=2^{\alpha}5^{\beta}19^{\gamma}$.
The signature $(0;2,5,17)$ is produced for $|G|=\frac{340}{41}(g-1)$, so $41$ divides $(g-1)$. By the gpnf-property $g-1=2^{\alpha}5^{\beta}17^{\gamma}41^{\delta}$.
The signature $(0;2,5,18)$ is produced for $|G|=\frac{90}{11}(g-1)$, so $11$ divides $(g-1)$. By the gpnf-property $g-1=2^{\alpha}3^{\beta}5^{\gamma}11^{\delta}$.
The signature $(0;2,5,19)$ is produced for $|G|=\frac{380}{47}(g-1)$, so $47$ divides $(g-1)$. By the gpnf-property $g-1=2^{\alpha}5^{\beta}19^{\gamma}47^{\delta}$.

(4) For signature of the form $(0;3,3,n)$, where $5 \leq n < 12$, one has that $8(g-1) <|G|=\left(\frac{6n}{n-3}\right)(g-1)<24(g-1)$.

The signature $(0;3,3,5)$ is produced for $|G|=15(g-1)$. By the gpnf-property $g-1=3^{\alpha}5^{\beta}$.
The signature $(0;3,3,6)$ is produced for $|G|=12(g-1)$. By the gpnf-property $g-1=2^{\alpha}3^{\beta}$.
The signature $(0;3,3,7)$ is produced for $|G|=\frac{21}{2}(g-1)$, so $2$ divides $(g-1)$. By the gpnf-property $g-1=2^{\alpha}3^{\beta}7^{\gamma}$.
The signature $(0;3,3,8)$ is produced for $|G|=\frac{48}{5}(g-1)$, so $5$ divides $(g-1)$. By the gpnf-property $g-1=2^{\alpha}3^{\beta}5^{\gamma}$.
The signature $(0;3,3,9)$ is produced for $|G|=9(g-1)$. By the gpnf-property $g-1=3^{\alpha}$.
The signature $(0;3,3,10)$ is produced for $|G|=\frac{60}{7}(g-1)$, so $7$ divides $(g-1)$. By the gpnf-property $g-1=2^{\alpha}3^{\beta}5^{\gamma}7^{\delta}$.
The signature $(0;3,3,11)$ is produced for $|G|=\frac{33}{4}(g-1)$, so $4$ divides $(g-1)$. By the gpnf-property $g-1=2^{\alpha}3^{\beta}11^{\gamma}$.

(5) For signature of the form $(0;3,4,n)$, where $n \in \{4,5\}$, one has that $8(g-1) <|G|=\left(\frac{24n}{5n-12}\right)(g-1)<24(g-1)$.

The signature $(0;3,4,4)$ is produced for $|G|=12(g-1)$. By the gpnf-property $g-1=2^{\alpha}3^{\beta}$.
The signature $(0;3,4,5)$ is produced for $|G|=\frac{120}{13}(g-1)$, so $13$ divides $(g-1)$. By the gpnf-property $g-1=2^{\alpha}3^{\beta}5^{\gamma}13^{\delta}$.

%%%%%%%%%%%%%%%%%%%%%%%
%%%%%%%%%%%%%%%%%%%%%%%%
\section{Proof of Theorem \ref{teo2}}
%%%%%%%%%%%%%
\subsection{Proof of part (1): $n=3$}
Let us observe that the Riemann surface $S$ of genus two defined by the algebraic equation ($a\in {\mathbb C}\setminus \{0,1,\omega_{3}, \omega_{3}^{2}\}$, where $\omega_{3}=e^{2\pi i/3}$)
$$S: W^{2}=(U^{3}-1)(U^{3}-a^{3})$$
admits the group of conformal automorphisms
$$\langle x(U,W)=(\omega_{3}U,W), \; y(U,W)=(\frac{a}{U},\frac{W}{U^{3}})\rangle \cong {\mathbb D}_{3},$$
which acts gpnf. So, $\sigma^{pnf}_{hyp}({\mathbb D}_{3})=2$.

%%%%%%%%%%%%%%
\subsection{Proof of parts (2) and (3): Case $n \geq 4$}
Let $n=p_{1}^{\alpha_{1}}q \geq 3$, where $p_{1}$ is the smallest prime divisor of $n$, $\gcd(p_{1},q)=1$, and $\alpha_{1} \geq 1$.

Let us recall that ${\mathbb D}_{n}=G^{0}_{n-1}(n)=\langle x, y : x^{n}=y^{2}=(xy)^{2}=1\rangle$, and 

$$\sigma^{0}_{hy}({\mathbb D}_{n})=\left\{\begin{array}{ll}
n+1-p_{1}-\dfrac{n}{p_{1}}, & \alpha_{1}=1, \; q \geq 2.\\
n-\dfrac{n}{p_{1}}, & \mbox{if either $\alpha_{1}\geq 2$, or $\alpha_{1}=q=1$}.
\end{array}
\right.
$$

Moreover, if $S$ is a closed Riemann surface admitting ${\mathbb D}_{n}$ as a group of conformal automorphisms, and $S$ has the minimal genus $\sigma^{0}_{hy}({\mathbb D}_{n})$, then $S/{\mathbb D}_{n}$ has signature:
$$\left\{\begin{array}{ll}
\left(0;2,2,p_{1},\frac{n}{p_{1}}\right), & \alpha_{1}=1, \; q \geq 2.\\
\left(0;2,2,p_{1},n\right), & \mbox{if either $\alpha_{1}\geq 2$, or $q=1$}.
\end{array}
\right\}.
$$

The maximal cyclic subgroups of ${\mathbb D}_{n}$ are $\langle x \rangle \cong {\mathbb Z}_{n}$, and the order two cyclic groups $\langle x^{l}y\rangle$, where $l \in \{0,1,\ldots,n-1\}$. Up to conjugation, one has the following maximal cyclic groups:
\begin{enumerate}[leftmargin=15pt]
\item $\langle x \rangle$ and $\langle y \rangle$, if $n$ is odd; and
\item $\langle x \rangle$, $\langle y \rangle$ and $\langle xy \rangle$, if $n$ is even.
\end{enumerate}

In particular, if ${\mathbb D}_{n}$ acts gpnf on a Riemann surface of genus $g \geq 2$, then (by Proposition \ref{propo1}) the signature of $S/{\mathbb D}_{n}$ must be of the form
$$(\gamma;n,\stackrel{d_{1}}{\cdots},n,2,\stackrel{d_{2}}{\cdots},2,k_{1},\ldots,k_{r}),$$
where $d_{1},d_{2} \geq 1$, and $k_{i} \geq 2$ are odd divisors of $n$, and (by the Riemann-Hurwitz formula)
$$\frac{1}{n} \leq \frac{g-1}{n}=2(\gamma-1)+d_{1}(1-\frac{1}{n})+\frac{d_{2}}{2}+\sum_{j=1}^{r}\left(1-\frac{1}{k_{j}}\right).$$

%%%%%%%%%%%%%%
\subsubsection{} Let us first consider either of the following cases:
\begin{enumerate}[leftmargin=15pt]
\item[(i)] $\alpha_{1} \geq 2$;
\item[(ii)] $\alpha=1$ and $q=1$;
\item[(iii)] $\alpha_{1}=1$, $q \geq 2$ and $n \geq 3$ odd.
\end{enumerate}

Let us consider a Fuchsian group with signature $(0;n,p_{1},2,2)$
\begin{equation}\label{Delta1}
\Delta_{1}=\langle x_{1},x_{2},x_{3},x_{4}:x_{1}^{n}=x_{2}^{p_{1}}=x_{3}^{2}=x_{4}^{2}=x_{1}x_{2}x_{3}x_{4}=1\rangle,
\end{equation}
and the surjective homomorphism
$\theta_{1}:\Delta_{1} \to {\mathbb D}_{n}$,
defined by $\theta_{1}(x_{1})=x, \; \theta_{1}(x_{2})=x^{m}, \; \theta_{1}(x_{3})=y, \; \theta_{1}(x_{4})=x^{1+m}y$ (where $m=n/p_{1}=p_{1}^{\alpha_{1}-1}q$).
Since $\theta_{1}$ preserves the orders, its kernel $\Gamma_{1} \lhd \Delta_{1}$ is torsion-free. So, $S_{1}={\mathbb H}^{2}/\Gamma_{1}$ is closed Riemann surface admitting ${\mathbb D}_{n}$ as group of conformal automorphisms and with $S_{1}/{\mathbb D}_{n}$ of signature $(0;n,p_{1},2,2)$. So, by the Riemann-Hurwitz formula, the genus of $S_{1}$ is $\sigma^{0}_{hyp}({\mathbb D}_{n})$.

Let us first assume $n \geq 3$ is odd. In this case, all the maximal cyclic subgroups of ${\mathbb D}_{n}$, different from $\langle x \rangle$, are conjugate to $\langle y \rangle$. As a consequence of Lemma \ref{lemita1}, the above conformal action is gpnf.
Next, let us assume $n \geq 4$ is even. In this case, $p_{1}=2$. Since $\alpha_{1} \geq 2$, $m=n/p_{1}$ is also even. This asserts that $x^{1+m}y$ is conjugate to $xy$. So, again by Lemma \ref{lemita1}, the conformal action is gpnf.

%%%%%%%%%%%%%%
\subsubsection{}
Let us now assume $n \geq 4$ even and $\alpha_{1}=1$; so $p_{1}=2$ and $q\geq 3$ odd. Let $q_{1}$ the smallest prime divisor of $q$. Let us now consider a Fuchsian group of signature $(0;2,2,q_{1},n)$
\begin{equation}\label{Delta2}
\Delta_{2}=\langle x_{1},x_{2},x_{3},x_{4}:x_{1}^{n}=x_{2}^{q_{1}}=x_{3}^{2}=x_{4}^{2}=x_{1}x_{2}x_{3}x_{4}=1\rangle,
\end{equation}
and the surjective homomorphism (where $m=n/q_{1}$)
$\theta_{2}:\Delta_{2} \to {\mathbb D}_{n}$,
defined by $$\theta_{2}(x_{1})=x, \; \theta_{2}(x_{2})=x^{m}, \; \theta_{2}(x_{3})=y, \; \theta_{2}(x_{4})=x^{1+m}y.$$

As before, the kernel $\Gamma_{2} \lhd \Delta_{2}$ of $\theta_{2}$ is torsion-free. So, $S_{2}={\mathbb H}^{2}/\Gamma_{2}$ is closed Riemann surface admitting ${\mathbb D}_{n}$ as group of conformal automorphisms and with $S_{2}/{\mathbb D}_{n}$ of signature $(0;2,2,q_{1},n)$. So, by the Riemann-Hurwitz formula, the genus of $S_{2}$ is $n-n/q_{1}$. As $1+n/q_{1}$ is odd, $x^{1+n/q_{1}}y$ is conjugate to $xy$. So, by Lemma \ref{lemita1}, the conformal action is gpnf.

To finish, we only need to observe that there is no gpnf-action of ${\mathbb D}_{n}$ with signature $(0;2,2,p_{1},n)$ (as this is the only possible signature of a gpnf-action for ${\mathbb D}_{n}$ with lower hyperbolic area). This follows from the fact that for any surjective homomorphism $\theta:\Delta_{1} \to {\mathbb D}_{n}$ there is an automorphism $\phi \in {\rm Aut}({\mathbb D}_{n})$ such that $\phi \circ \theta=\theta_{1}$. But in this case, $xy$ acts free fixed points.

%%%%%%%%%%%%%%%%
%%%%%%%%%%%%%%%%%
\section{Proof of Theorem \ref{teo0} and Corollary \ref{DCn}}
\subsection{Proof of Theorem \ref{teo0}}\label{Sec:prueba0}
We will follow the strategy as mentioned in Remark \ref{estrategia}.
As we already known the case for $G$ either ${\mathbb Z}_{n}$, ${\mathbb Z}_{2}^{2}$, ${\mathbb D}_{n}$, ${\mathcal A}_{4}$, and ${\mathcal A}_{5}$ and ${\mathfrak S}_{4}$ (in which case $\sigma^{pnf}(G)=0$), we only need to assume $G$ is different from any of these groups.

Set $L=L_{G}$. Let $2 \leq m_{1} \leq \cdots \leq m_{L} \leq n$ be the orders of a maximal collection of pairwise non-conjugate maximal cyclic subgroups.
By Proposition \ref{propo1}, every gpnf-action of $G$ must have quotient signature of the form
$$(*) \; (\gamma;m_{1},\ldots,m_{L},k_{1},\ldots,k_{r}).$$

Let us observe that each tuple $(B_{1},\ldots,B_{L}) \in {\mathcal F}$ permits us to construct (by following the same arguments as for the proof of Theorem \ref{teogpnf}) a gpnf-action of $G$ on a closed Riemann surface of genus 
$$g_{B}=1+n\left(L-1-\frac{1}{m_{B}}-\sum_{i=1}^{L}\frac{1}{|B_{i}|} \right),$$
whose quotient signature is $(0;m_{1},\ldots,m_{L},m_{B})$. We may observe that each possible signature of the form $(*)$ produces a genus bigger (if $\gamma \geq 1$ or if $\gamma=0$ and $r \geq 2$) than one of the values $g_{B}$. Now, as defined $B_{0}$, we may observe that $g_{B_{0}} \leq g_{B}$.

%%%%%%%%%%%%%%%%
\subsection{Proof of Corollary \ref{DCn}}
Let us recall that, for every integer $n \geq 2$, the dicyclic group of order $4n$ is the group with the following presentation:
$$DC_{n}=\langle x, y : x^{2n}=1, y^{2}=x^{n}, y^{-1}xy=x^{-1}\rangle.$$

Note that $DC_{2}=Q_{8}$. It follows from the above presentation that $|x|=2n$, $|y|=|xy|=4$, that every non-trivial element of $DC_{n}$ has the form $x^{l}$, where $1 \leq l \leq 2n-1$, or 
$x^{l}y$, where $0 \leq l \leq 2n-1$.  Since $x^{l}yx^{-l}=x^{2l}y$ and $x^{l} (xy) x^{-l}=x^{2l+1}y$, we may observe that 
every maximal cyclic subgroup is conjugated to either $C_{1}=\langle x \rangle \cong {\mathbb Z}_{2n}$ (a normal subgroup), $C_{2}=\langle y \rangle \cong {\mathbb Z}_{4}$, or $C_{3}=\langle xy \rangle \cong {\mathbb Z}_{4}$. Note that $C_{1}$ and $C_{2}$ are non-conjugate. If $n \geq 2$ is even, then $C_{1}$, $C_{2}$, and $C_{3}$ are non-conjugate, so $L_{DC_{n}}=3$. In $n \geq 3$ is odd, then $C_{2}$ and $C_{3}$ are conjugate, so $L_{DC_{n}}=2$.
Since $x \cdot y \cdot (xy)^{-1}=1$, we observe that, for $n$ even one has that $m_{DC_{n}}=1$, and for $n$ odd, one has that $m_{DC_{n}}=4$.
As a consequence of Theorem \ref{teo0}, $\sigma^{pnf}_{hyp}({DC_{n}})=n$. The quotient signature is $(0;2n,4,4)$.

%%%%%%%%%%%%%%%%%%%%%%%
%%%%%%%%%%%%%%%%%%%%%%%%
\section{Proof of Theorem \ref{teo3}}
Let us recall that we are assuming $n\geq 3$,  and $2 \leq s \leq n-2$ be such that $s^{2} \equiv 1 \mod(n)$. 
The condition $2 \leq s \leq n-2$ asserts that $n \geq 4$ is not a prime integer. Let us also recall that in this case
$$\sigma^{0}_{hyp}(G^{0}_{s}(n))=\frac{n-\gcd(n,s+1)}{2},$$
and if $S$ is a closed Riemann surface of that genus admitting a conformal action by $G^{0}_{s}(n)$, then $S/G^{0}_{s}(n)$ has signature 
$$\left(0;2,\frac{2n}{\gcd(n,s+1)},n\right).$$

One such minimal action is produced by the following surjective homomorphism with torsion-free kernel:
$$\theta:\Delta=\langle x_{1},x_{2},x_{3}:x_{1}^{n}=x_{2}^{2}=x_{3}^{\frac{2n}{\gcd(n,s+1)}}=x_{1}x_{2}x_{3}=1\rangle \to G^{0}_{s}(n),$$
$$\theta(x_{1})=x, \; \theta(x_{2})=y, \; \theta(x_{3})=x^{-s}y.$$

%%%%%%%%%%%%%%%
\subsection{Topological rigidity}
Let us observe that every conformal action of the group $G^{0}_{s}(n)$ on a closed Riemann surface of genus $\sigma^{0}_{hyp}(G^{0}_{s}(n))$ is determined by a surjective homomorphism $\eta:\Delta \to G^{0}_{s}(n)$, with a torsion-free kernel. To obtain the topological rigidity property, we need to find $\phi \in {\rm Aut}(G^{0}_{s}(n))$ and $\psi \in {\rm Aut}(\Delta)$ such that $\theta=\phi \circ \eta \circ \psi$.
It follows that  $\eta(x_{1})$ has order $n$, that $\eta(x_{2})$ has order two (which, by surjectivity of $\eta$, cannot belong to $\langle x \rangle$). In this way, 
\begin{enumerate}[leftmargin=15pt]
\item $\eta(x_{2})=x^{v}y$, where $v \in \{0,1,\ldots,n-1\}$, and $v(s+1) \equiv 0 \mod(n)$, and 
\item (by Lemma \ref{lemita2}) either (i) $\eta(x_{1})=x^{l}$ or (ii) $\eta(x_{1})=x^{l}y$, where $l \in \{1,\ldots,n-1\}$ and $\gcd(n,l)=1$.
\end{enumerate}

%%%%%%%%%
\subsubsection{}
Let us first consider the case that $\eta(x_{1})=x^{l}$, where $l \in \{1,\ldots,n-1\}$ and $\gcd(n,l)=1$. In this case, if we consider $\psi=I$, and $\phi \in {\rm Aut}(G^{0}_{s}(n))$ defined by $\phi(x)=x^{l}, \; \phi(y)=x^{v}y,$
then $(\phi^{-1} \circ \eta)(x_{1})=x$, and $(\phi^{-1} \circ \eta)(x_{2})=y$, and we are done.

%%%%%%%%%
\subsubsection{}
Let us now consider the case that $\eta(x_{1})=x^{l}y$, where $l \in \{1,\ldots,n-1\}$ and $\gcd(n,l)=1$.  In this case, if we consider $\psi=I$, and $\phi \in {\rm Aut}(G^{0}_{s}(n))$ defined by $\phi(x)=x^{l}y, \; \phi(y)=x^{v}y,$
then $(\phi^{-1} \circ \eta)(x_{1})=x$, and $(\phi^{-1} \circ \eta)(x_{2})=y$, and we are done.

%%%%%%%%%%%%%%%
\subsection{Algebraic models}
Let $S$ be a closed Riemann surface of genus $\sigma^{0}_{hyp}(G^{0}_{s}(n))$ admitting the conformal action of $G^{0}_{s}(n)$. Since $S/G^{0}_{s}(n)$ has signature $\left(0;2,\frac{2n}{\gcd(n,s+1)},n\right)$, and the orbifold $S/\langle x \rangle$ admits a conformal involution $\hat{y}$ such that $S/G^{0}_{s}(n)=(S/\langle x \rangle)/\langle \hat{y}\rangle$, we may observe that 
$S/\langle x \rangle$ has signature $(0;n,n,\frac{n}{\gcd(n,s+1)})$, where the cone point of order $\frac{n}{\gcd(n,s+1)}$ is fixed by $\hat{y}$, and the two of order $n$ are permuted by it. So, we may identify $S/\langle x \rangle$ with the Riemann sphere $\widehat{\mathbb C}$ and its cone points to be $0$ (the one of order $\frac{n}{\gcd(n,s+1)}$), and $\pm 1$ (the two of order $n$. In this case, $\hat{y}(U)=-U$. This permits us to see that $S$ is a ${\mathbb Z}_{n}$-cyclic gonal curve and that an algebraic model for it has the form
$$W^{n}=U^{\gcd(n,s+1)}(U-1)^{\alpha}(U+1)^{s\alpha},$$
where $\alpha \in \{1,\ldots,n-1\}$ is such that $\gcd(n,\alpha)=1$ and $\gcd(n,s+1)+\alpha(s+1) \equiv 0 \mod(n)$.

Let us write $s+1=t \gcd(n,s+1)$, where $\gcd(n,t)=1$. If $t\beta \equiv 1 \mod(n)$, then $\alpha=n-\beta$ works.
Now, the choices of $x(U,W)$ and $y(U,E)$ as stated in the theorem are seen to be automorphisms of the above algebraic curve and that they generate $G^{0}_{s}(n)$.

%%%%%%%%%%%%%%%
\subsection{The case $n$ odd: $\sigma^{pnf}_{hyp}(G^{0}_{s}(n))=\sigma^{0}_{hyp}(G^{0}_{s}(n))$}
Now, under the assumption that $n$ is odd, we proceed to see that the conformal action of $G^{0}_{s}(n)$ on the Riemann surface $S$ as before is gpnf. 
Let us observe that this case is also proved by Proposition \ref{casosimple} once we have that $\langle x \rangle$, $\langle y \rangle$ and $\langle xy \rangle$ form a maximal collection of pairwise non-conjugate maximal cyclic subgroups. 
We already know that $x$ has fixed points. So, we need to check that every element of the form $x^{l}y$ has a fixed point.  
Let us denote by $\pi_{n}:S \to S/\langle x \rangle$ the branched regular covering with deck group $\langle x \rangle$.  As above, let $0$ be the branch value of $\pi_{n}$ of order $\frac{n}{\gcd(n,s+1)}$.
Since $F:=\pi_{n}^{-1}(0)$ is the set of fixed points of $x^{\gcd(n,s+1)}$, it is invariant under $\langle x \rangle$. As $\hat{y}(0)=0$, it follows that this set is also invariant under $y$. So, $F$ is invariant under $G^{0}_{s}(n)$. The cardinality of $F$ is $\gcd(n,s+1)$. So, as $n$ is odd, $F$ has odd cardinality. 
As $x^{l}y$ has even order, say $\frac{2n}{\gcd(n,s+1)}$, and it keeps $F$ invariant, necessarily it must have a fixed point in $F$.

%%%%%%%%%%%%%%%%
\subsection{The case $n$ even in part (3)}
This is a consequence of Proposition \ref{casosimple}.

%%%%%%%%%%%%%%%%
%%%%%%%%%%%%%%%%
\section{GAP implementation}\label{Sec:GAP}
If $G$ is neither a cyclic group, nor ${\mathbb Z}_{2}^{2}$, nor a dihedral group, nor ${\mathcal A}_{4}$, nor ${\mathcal A}_{5}$, nor ${\mathfrak S}_{4}$, then 
$\sigma^{pnf}_{hyp}(G)=\sigma^{pnf}(G)$. For the left groups $\sigma^{pnf}(G)=0$.
Theorem \ref{teo0} permits us to implement the computation of $\sigma^{pnf}(G)$ in the GAP computational Program \cite{GAP}. 

\medskip
\noindent
{\small
\fbox{\parbox{\textwidth}{
{\bf The function ${\rm gpnf}(G)$:} For each finite group $G$, it is computed the parameter $\sigma^{gpnf}(G)$ together with the quotient signature of that gpnf-action.

\medskip\hrule\medskip

{\bf Input:} The finite group $G$.

{\bf Output:} $\sigma^{pnf}(G)$ and the quotient signature.

{\bf The algorithm:}
\begin{enumerate}[leftmargin=15pt]
\item Compute the maximal cyclic subgroups of $G$.
\item Make a list $C=[C_{1},\ldots,C_{L_{G}}]$, whose elements $C_{i}$ are the lists of the conjugacy classes of maximal cyclic subgroups.
\item Make the list $\Delta=[m_{1},\ldots,m_{L_{G}}]$, where $m_{j}$ is the order of a maximal cyclic subgroup contained in $C_{i}$.
\item Make the list $W=[W_{1},\ldots,W_{L_{G}}]$, where each $W_{i}$ is the list of all generators of all the maximal cyclic subgroups in $C_{i}$.
\item For each $b=[w_{1},\ldots,w_{L_{G}}]$, where $w_{i} \in W_{i}$, compute the order $m_{b}$ of $w_{1}\cdots w_{L_{G}}$.
\item Compute the minimal value $m_{G}$ of $m_{b}$, for $b$ running on all the above tuples.
\item If $m_{G} \geq 2$,  then add $m_{G}$ to the list $\Delta$.
\item Use the signature $(0;m_{1},\ldots,m_{L_{G}},m_{G})$ and the Riemann-Hurwitz formula to obtain $\sigma^{pnf}(G)$.
\end{enumerate}
}}
}

%%%%%%%%%%%%%%%
%%%%%%%%%%%%%%%%
\section{Some examples}\label{Sec:ejemplos}
%%%%%%%%%%%%%%%%
\subsection{Examples for groups of order $8$}
There are, up to isomorphism, five groups of order $8$: $G_{1}={\mathbb Z}_{8}$, $G_{2}={\mathbb Z}_{2} \times {\mathbb Z}_{4}$, $G_{3}={\mathbb D}_{4}$, $G_{4}=Q_{8}$ and $G_{5}={\mathbb Z}_{2}^{3}$. For these groups, the $\sigma^{pnf}_{hyp}$-values are given as follows.
(1) $\sigma^{pnf}_{hyp}({\mathbb Z}_{8})=2$, and quotient signature $(0;2,8,8)$.
(2) $\sigma^{pnf}_{hyp}({\mathbb Z}_{2} \times {\mathbb Z}_{4})=5$, and quotient signature $(0;2,2,2,4,4)$.
(3) $\sigma^{pnf}_{hyp}({\mathbb D}_{4})=2$, and quotient signature $(0;2,2,2,4)$. 
(4) $\sigma^{pnf}_{hyp}(Q_{8})=2$, and quotient signature $(0;4,4,4)$. 
(5) $\sigma^{pnf}_{hyp}({\mathbb Z}_{2}^{3})=7$, and quotient signature $(0;2,2,2,2,2,2,2)$.

%%%%%%%%%%%%%%%%%%%%%%%
\subsection{The groups ${\mathfrak S}_{5}$ and ${\mathcal A}_{6}$}
It is known that $\sigma^{0}_{hyp}({\mathfrak S}_{5})=4$ (corresponding to a conformal action of the so-called Bring's curve \cite{Edge})
This group has $31$ maximal cyclic subgroups, which form $3$ conjugacy classes whose representatives have orders $4$, $5$, and $6$, respectively.
By applying Theorem \ref{teo0} to this group, we obtain that $\sigma^{pnf}_{hyp}({\mathfrak S}_{5})=24$, and with quotient signature $(0;4,5,6)$.
It is known that $\sigma^{0}_{hyp}({\mathcal A}_{6})=10$ \cite{Nathan})
This group has $121$ maximal cyclic subgroups, which form $4$ conjugacy classes whose representatives have orders $3$, $3$, $4$ and $5$, respectively.
By applying Theorem \ref{teo0} to this group, we obtain that $\sigma^{pnf}_{hyp}({\mathcal A}_{6})=160$, and with quotient signature $(0;3,3,4,5)$.

%%%%%%%%%%%%%%%%%%%%%%%
\subsection{The groups with strong symmetric genus two and twenty-four}
In \cite{Nathan}, there are computed the groups $G$ for which $\sigma^{0}(G) \in \{2,\ldots,24\}$. If we want to compute those groups $G$ with $\sigma^{pnf}(G)=g$, then (as $\sigma^{0}(G) \leq \sigma^{pnf}(G)$) we only need to compute the values of $\sigma^{pnf}(K)$ for those with $\sigma^{0}(K) \leq g$. For instance, the following tables consider all those groups $G$ with $\sigma^{0}(G) \in\{2,24\}$.
{\small
$$
(\sigma^{0}(G)=2)\;\;
\begin{array}{|c|c|c|c|}\hline
\mbox{Structure Description of $G$} & \mbox{GAP ID} & \sigma^{pnf}(G) & \mbox{Quotient Signature} \\\hline\hline 
Q_{8} & [8,4] & 2 & (0;4,4,4) \\\hline 
DC_{3} & [12,1] & 3 & (0;4,4,6) \\\hline 
QD_{16} & [16,8] & 2 & (0;2,4,8) \\\hline 
SL(2,3) & [24,3] & 4 & (0;3,4,6) \\\hline 
({\mathbb Z}_{6} \times {\mathbb Z}_{2}) \rtimes {\mathbb Z}_{2} & [24,8] & 12 & (0;2,4,6,6) \\\hline 
GL(2,3) & [48,29] & 6 & (0;2,6,8) \\\hline 
\end{array}
$$
}

{\small
$$
(\sigma^{0}(G)=24)\;\;
\left.\begin{array}{|c|c|c|c|}\hline 
\mbox{Structure Description of $G$} & \mbox{GAP ID} & \sigma^{pnf}(G) & \mbox{Quotient Signature}  \\\hline\hline
 {\mathbb Z}_{9} \rtimes {\mathbb Z}_{8} & [72,1] & 27 & (0;8,8,36) \\\hline
  {\mathbb Z}_{5} \rtimes Q_{16} & [80,18] & 62 & (0;4,8,20,20) \\\hline 
  DC_{24} & [96,8] & 24 & (0;4,4,48) \\\hline 
  {\mathbb Z}_{25} \rtimes {\mathbb Z}_{4} & [100,1] & 25 & (0;4,4,50) \\\hline 
  {\mathbb Z}_{17} \times {\mathfrak S}_{3} & [102,1] & 24 & (0;2,34,51) \\\hline 
  {\mathbb Z}_{13} \times {\mathbb D}_{4} & [104,10] & 48 & (0;26,26,52) \\\hline 
  {\mathbb Z}_{2}^{2} \rtimes {\mathbb Z}_{27} & [108,3] & 48 & (0;18,27,27) \\\hline 
  {\mathbb Z}_{7} \times {\mathbb D}_{8} & [112,24] & 48 & (0;14,14,56) \\\hline 
  {\mathbb Z}_{3} \times (({\mathbb Z}_{10} \times {\mathbb Z}_{2}) \rtimes {\mathbb Z}_{2}) & [120,20] & 102 & (0;6,12,30,30) \\\hline 
  {\mathbb Z}_{5} \times {\mathbb D}_{12} & [120,23] & 48 & (0;10,10,60) \\\hline 
  {\mathbb Z}_{5} \times {\mathfrak S}_{4} & [120,37] & 48 & (0;10,15,20) \\\hline 
  {\mathbb D}_{5} \times {\mathbb D}_{7} & [140,7] & 92 & (0;2,10,14,35) \\\hline 
  ({\mathbb Z}_{9} \times {\mathbb D}_{4}) \rtimes {\mathbb Z}_{2} & [144,16] & 94 & (0;2,8,18,36) \\\hline 
  {\mathfrak S}_{3} \times {\mathbb D}_{13} & [156,11] & 100 & (0;2,6,26,39) \\\hline 
  {\mathbb Z}_{96} \rtimes {\mathbb Z}_{2} & [192,8] & 24 & (0;2,4,96) \\\hline   
  ({\mathbb Z}_{50} \times {\mathbb Z}_{2}) \rtimes {\mathbb Z}_{2} & [200,8] & 122 & (0;2,4,50,50) \\\hline 
  ({\mathbb Z}_{2}^{2} \rtimes {\mathbb Z}_{27}) \rtimes {\mathbb Z}_{2} & [216,21] & 126 & (0;2,4,18,27) \\\hline 
  \end{array}\right.
$$
}

\begin{rema}
Our computations also permit us to observe the following.
\begin{enumerate}
\item There are two groups $G$ with $\sigma^{pnf}(G)=2$: ${\rm SmallGroup}(8,4)=Q_{8}$ and ${\rm SmallGroup}(16,8)=QD_{16}$.
\item There is only one group $G$ with $\sigma^{pnf}(G)=3$: ${\rm SmallGroup}(12,1)=DC_{3}$.
\item There are three groups $G$ with $\sigma^{pnf}(G)=4$: ${\rm SmallGroup}(24,3)=SL(2,3)$, ${\rm SmallGroup}(16,9)=Q_{16}$ and ${\rm SmallGroup}(32,19)=QD_{32}$.
\item There are two groups $G$ with $\sigma^{pnf}(G)=5$: ${\rm SmallGroup}(20,1)=DC_{5}$ and ${\rm SmallGroup}(30,2)={\mathbb Z}_{3} \times {\mathbb D}_{5}$.

\end{enumerate}
\end{rema}

%%%%%%%%%%%%%%%%%
\subsection{Some examples for $G^{0}_{s}(n)$}
Let $n \geq 3$ and $2 \leq s \leq n-2$ be integers such that $s^{2} \equiv 1 \mod(n)$. 

\begin{rema}
If $n=2^{a}p_{1}^{m_{1}}\cdots p_{r}^{m_{r}}$ is the prime decomposition of $n$, then the Chinese Remainder Theorem asserts that the number of solutions $s \in \{1,\ldots,n-1\}$ of the previous equation is equal to $2^{r}\alpha_{a}$, where $\alpha_{0}=\alpha_{1}=1$, $\alpha_{2}=2$, and $\alpha_{a}=4$ for $a \geq 3$. Since $1$ and $n-1$ are two of those solutions, the number of those solutions in $\{2,\ldots,n-1\}$ is $2^{r}\alpha_{a}-2$.
\end{rema}

%%%%%%%%%%%%
\subsubsection{}
If $n \geq 3$ is odd, then part (2) of Theorem \ref{teo3} asserts the equality $\sigma^{0}_{hyp}(G^{0}_{s}(n)=\sigma^{pnf}_{hyp}(G^{0}_{s}(n))$.

%%%%%%%%%%%%
\subsubsection{}
If $n \geq 8$ is even, then part (3) of Theorem \ref{teo3} asserts that the equality $\sigma^{0}_{hyp}(G^{0}_{s}(n)=\sigma^{pnf}_{hyp}(G^{0}_{s}(n))$ holds if and only if every maximal cyclic subgroup of $G^{0}_{s}(n)$ is conjugate to either $C_{1}=\langle x \rangle$, $C_{2}=\langle y \rangle$ or $C_{3}=\langle xy\rangle$.
In case the above situation does not hold, we may use Theorem \ref{teo0} to obtain the value of $\sigma^{pnf}_{hyp}(G^{0}_{s}(n))$ (see the algorithm in Section \ref{Sec:GAP}).
 In Table \ref{tablita1} are shown some cases (the computations have been carried out using GAP \cite{GAP}).
Below, we work out the first case $n=8$.

%%%%%%%%%%%%%%%%
{\tiny
\begin{table}[h]
$$
\begin{array}{|c|c|c|c|c|}\hline n & s & gpnf/non-gpnf & \sigma^{0}_{hyp}(G^{0}_{s}(n)) & \sigma^{0}_{hyp}(G^{0}_{s}(n))\\\hline\hline 
8 & 3 & gpnf & 2& 2\\\hline 
8 & 5 & non-gpnf & 3 &13  \\\hline 
10 & \emptyset & {} & & \\\hline
12 & 5 & non-gpnf & 3 &15  \\\hline 
12 & 7 & non-gpnf & 4 & 8 \\\hline 
14 & \emptyset & {} & & \\\hline
16 & 7 & gpnf & 4&4 \\\hline 
16 & 9 & non-gpnf & 7 & 33 \\\hline
18 & \emptyset & & & \\\hline
20 & 9 & non-gpnf & 5 & 25 \\\hline 
20 & 11 & non-gpnf & 8 & 16 \\\hline 
22 & \emptyset & {} & & \\\hline
24 & 5 & non-gpnf & 9 & 39 \\\hline 
24 & 7 & non-gpnf & 8 & 16 \\\hline 
24 & 11 & gpnf & 6 & 6 \\\hline 
24 & 13 & non-gpnf & 11 & 53 \\\hline 
24 & 17 & non-gpnf & 9 & 39  \\\hline 
24 & 19 & non-gpnf & 10 &  18\\\hline
\end{array}
\begin{array}{|c|c|c|c|c|}\hline n & s & gpnf/non-gpnf &\sigma^{0}_{hyp}(G^{0}_{s}(n))  & \sigma^{0}_{hyp}(G^{0}_{s}(n))    \\\hline\hline 
26 & \emptyset &  & & \\\hline 
28 & 13 & non-gpnf & 7 & 35 \\\hline 
28 & 15 & non-gpnf & 12 & 24 \\\hline 
30 & 11 & non-gpnf & 12 & 24 \\\hline 
30 & 19 & non-gpnf & 10 & 20 \\\hline 
32 & 15 & gpnf & 8 & 8 \\\hline 
32 & 17 & non-gpnf & 15 & 113 \\\hline 
34 & \emptyset &  & & \\\hline 
36 & 17 & non-gpnf & 9 & 45\\\hline 
36 & 19 & non-gpnf & 16 & 32\\\hline 
38 & \emptyset &  &  & \\\hline 
40 & 9 & non-gpnf &15 & 65 \\\hline 
40 & 11 & non-gpnf & 18 & 34 \\\hline 
40 & 19 & gpnf & 10 &10 \\\hline 
40 & 21 & non-gpnf & 19 & 93 \\\hline 
40 & 29 & non-gpnf & 15 & 65 \\\hline 
40 & 31 & non-gpnf & 16 & 32 \\\hline 
 & & & &\\\hline
\end{array}
$$
\label{tablita1}
\caption{Cases: $8 \leq n \leq 40$ is even and $2 \leq s \leq n-2$ is such that $s^{2} \equiv 1 \mod(n)$.}
\end{table}
}

%%%%%%%%%%%%
\subsubsection{Example: $n=8$}
 In this case, $s \in \{3,5\}$. 
\subsubsection*{\rm (1)}
If $s=3$, then (up to conjugation) the maximal cyclic subgroups of $G^{0}_{3}(8)$ are $C_{1}$, $C_{2}$, and $C_{3}$, so $\sigma^{pnf}_{hyp}(G^{0}_{3}(8))=\sigma^{0}_{hyp}(G^{0}_{3}(8))=2$, and quotient signature $(0;2,4,8)$.

\subsubsection*{\rm (2)} If $s=5$, then the  maximal cyclic subgroups (up to conjugation) are 
$C_{1}$, $C_{2}$, $C_{3}$, and $C_{4}=\langle x^{2}y\rangle$, so 
$\sigma^{pnf}_{hyp}(G^{0}_{5}(8))>\sigma^{0}_{hyp}(G^{0}_{5}(8))=3$.
To obtain a gpng action by $G^{0}_{5}(8)$, the quotient signature must be of the form $(\gamma;2,4,8,8,k_{1},\ldots,k_{r})$ (by Propositon \ref{propo1}).
Let us consider a Fuchsian group of signature $(0;2,4,8,8)$ (which is of the minimal hyperbolic area among all the Fuchsian groups with the above types of signatures)
$\Delta_{0}=\langle x_{1},x_{2},x_{3},x_{4},x_{5}: x_{1}^{8}=x_{2}^{2}=x_{3}^{2}=x_{4}^{8}=x_{5}^{4}=x_{1}x_{2}x_{3}=1\rangle,$
and the surjective homomorphism with torsion-free kernel
$\theta_{0}:\Delta_{0} \to G^{0}_{5}(8),$ where
$\theta_{0}(x_{1})=x, \; \theta_{0}(x_{2})=y, \; \theta_{0}(x_{3})=y, \; \theta_{0}(x_{4})=xy, \; \theta_{0}(x_{5})=x^{6}y.$
The conformal action of $G^{0}_{5}(8)$ induced by the homomorphisms $\theta_{0}$ is gpnf on a genus $g=13$ Riemann surface. So, in this case, 
$\sigma^{pnf}_{hyp}(G^{0}_{5}(8))=13$, and with quotient signature $(0;2, 2, 4, 8, 8)$.

%%%%%%%%%%%%%%%%%%
%%%%%%%%%%%%%%%%%
\bigskip

%\noindent{\bf Declarations.}
%\medskip

\noindent{\bf Data Availability Statement.} We have used GAP (Groups, Algorithms, and Programming \cite{GAP}) to implement the algorithm to compute $\sigma^{pnf}_{hyp}(G)$.

\medskip

\noindent{\bf Conflicts of interest.} The authors have no conflicts of interest to declare that are relevant to this article

\medskip

\noindent{\bf Declaration of generative AI and AI-assisted technologies in the manuscript preparation process}
During the preparation of this work, the authors used Grammarly to improve the English language. The authors reviewed and edited
the content as needed and take full responsibility for the content of the published article.

%%%%%%%%%%%%%%%%%%
%%%%%%%%%%%%%%%%%%%

\end{document}